\documentclass[12pt]{amsart}
\usepackage[english]{babel}
\usepackage{graphicx}
\usepackage{amsfonts}
\usepackage{amsmath}
\usepackage{amssymb}
\usepackage{amscd}
\usepackage[all,cmtip,matrix, arrow]{xy}

\def\Aut {\operatorname{Aut}}
\newcommand{\Sym}{\mathfrak{S}}
\newcommand{\Alt}{\mathfrak{A}}
\newcommand{\Cyc}{\mathfrak{C}}
\newcommand{\Dih}{\mathfrak{D}}
\newcommand{\PP}{\mathbb{P}}
\newcommand{\QQ}{\mathbb{Q}}

\newtheorem{lemma}{Lemma}[section]
\newtheorem{theorem}[lemma]{Theorem}
\newtheorem{propos}[lemma]{Proposition}

\newtheorem{corollary}[lemma]{Corollary}

\theoremstyle{definition}
\newtheorem{definition}[lemma]{Definition}

\newtheorem{remark}[lemma]{Remark}
\makeatletter
\def\blfootnote{\xdef\@thefnmark{}\@footnotetext}
\makeatother
\begin{document}

\title{Birational rigidity of $G$-del Pezzo threefolds of degree 2}
\author[A.\,A.~Avilov]{A.\,A.~Avilov}
\address{National Research University Higher School of Economics, AG Laboratory, HSE, 6 Usacheva str., Moscow, Russia, 119048.}
\email{v07ulias@gmail.com} 

\maketitle

\begin{abstract}
In this paper we classify nodal rational non-$\QQ$-factorial del Pezzo threefolds of degree $2$ which can be $G$-birationally rigid for some subgroup $G\subset \Aut(X)$.

\end{abstract}

\markright{Birational rigidity of $G$-del Pezzo threefolds of degree 2}

\blfootnote{}

\section{Introduction}

In this paper we work over an algebraically closed field $\operatorname{k}$ of characteristic 0. Recall that a \emph{$G$-variety} is a pair $(X, \rho)$, where $X$ is an algebraic variety and $\rho: G\to \Aut(X)$ is an injective homomorphism of groups. We say that $G$-variety $X$ has \emph{$G\mathbb{Q}$-factorial singularities} if every $G$-invariant Weil divisor of $X$ is $\mathbb{Q}$-Cartier.

Let $X$ be a $G$-variety with at most $G\mathbb{Q}$-factorial terminal singularities and $\pi:X\to Y$ be a $G$-equivariant morphism. We call $\pi$ a \emph{$G$-Mori fibration} if $\pi_{*}\mathcal{O}_{X}=\mathcal{O}_{Y}$, $\dim X>\dim Y$, the relative invariant Picard number $\rho^{G}(X/Y)$ is equal to $1$ (in this case we say that $G$ is \emph{minimal}) and the anticanonical class $-K_{X}$ is $\pi$-ample. If $Y$ is a point then $X$ is a \emph{$G\mathbb{Q}$-Fano variety}. If in addition the anticanonical class is a Cartier divisor then $X$ is a \emph{$G$-Fano variety}.

Let $X$ be arbitrary normal projective $G$-variety of dimension 3. Resolving the singularities of $X$ and running the $G$-equivariant minimal model program (see e.g.~\cite{Pro2}) we reduce $X$ either to a $G$-variety with nef anticanonical class, or to a $G$-Mori fibration (see e.g.~\cite[\S 3]{23}). So such fibrations (and $G\mathbb{Q}$-Fano varieties in particular) form a very important class in the birational classification. In this paper we consider a certain class of $G$-Fano threefolds.

\begin{definition} A projective $n$-dimensional variety $X$ is a \emph{del Pezzo variety} if it has at most terminal Gorenstein singularities and the anticanonical class $-K_{X}$ is ample and divisible by $n-1$ in the Picard group $\operatorname{Pic}(X)$. If a $G$-Fano variety $X$ is a del Pezzo variety, then we say that $X$ is a \emph{$G$-del Pezzo} variety.
\end{definition}

Del Pezzo varieties of arbitrary dimension were classified by T. Fujita (\cite{24}, \cite{25}, \cite{26}, see also~\cite{KuP}). $G\mathbb{Q}$-factorial $G$-minimal three-dimensional $G$-del Pezzo varieties were partially classified by Yu. Prokhorov in~\cite{Pro1}. The main invariant of a del Pezzo threefold $X$ is the \emph{degree} $d=(-\frac{1}{2}K_{X})^{3}$, it is an integer in the interval from 1 to 8. In this paper we consider the case $d=2$. In this case $X$ is a double cover of $\mathbb{P}^{3}$ with ramification at a quartic surface. If $d=8$ then $X$ is a projective space. In this case equivariant birational geometry were studied by I. Cheltsov and C. Shramov in the paper~\cite{57}. The cases $d=4$ and $d=3$ were considered in the author's previous works~\cite{Avi1} and~\cite{Avi2} (see also~\cite{CDK}). If $d>4$ then $X$ is smooth (cf.~\cite{Pro1}) while smooth del Pezzo threefolds and their automorphism groups are known well. For other types of $G$-Fano threefolds there are only some partial results: see for example~\cite{13},~\cite{17}.

Classification of finite subgroups of the Cremona group $\operatorname{Cr}_{3}(\operatorname{k})$ is one of the motivations of this paper. The Cremona group $\operatorname{Cr}_{n}(\operatorname{k})$ is the group of birational automorphisms of the projective space $\mathbb{P}^{n}_{\operatorname{k}}$. Finite subgroups of $\operatorname{Cr}_{2}(\operatorname{k})$ were comple\-tely classified by I. Dolgachev and V. Iskovskikh in~\cite{DI1}. The core of their method is the following. Let $G$ be a finite subgroup of $\operatorname{Cr}_{2}(\operatorname{k})$. The action of $G$ can be regularized in the following sence: there exists a smooth projective $G$-variety $Z$ and a birational morphism $Z\to \mathbb{P}^{2}$ which commutes with the action of $G$. Then we apply the equivariant minimal model program to $Z$ and obtain a $G$-Mori fibration which is either a $G$-conic bundle over $\mathbb{P}^{1}$, or a $G$-minimal del Pezzo surface. Dolgachev and Iskovskikh classified all minimal subgroups in automorphism groups of del Pezzo surfaces and conic bundles and so they obtained a complete list of finite subgroups of $\operatorname{Cr}_{2}(\operatorname{k})$. But quite often two subgroups from such list are conjugate in $\operatorname{Cr}_{2}(\operatorname{k})$, so it is natural to identify them. One can see that $G$-varieties $Z_{1}$ and $Z_{2}$ give us conjugate subgroups if and only if there exists a $G$-equivariant birational map $Z_{1}\dasharrow Z_{2}$. So we need to classify all rational $G$-Mori fibrations and birational maps between them as well.

Following this program in the 3-dimensional case one can reduce the question of classification of all finite subgroups in $\operatorname{Cr}_{3}(\operatorname{k})$ to the question of classification of all rational $G\mathbb{Q}$-Mori fibrations and birational equivariant maps between them. Such program was realized in some particular cases: simple non-abelian groups which can be embedded into $\operatorname{Cr}_{3}(\mathbb{C})$ (see~\cite{10}, see also~\cite{4},~\cite{14},~\cite{ChS},~\cite{CPS1}) and $p$-elementary subgroups of $\operatorname{Cr}_{3}(\mathbb{C})$ (see~\cite{11},~\cite{7}).

\begin{definition} By a \emph{node} we mean an isolated ordinary double singular point of a variety of arbitrary dimension. A variety is called \emph{nodal} if its only singularities are nodes.
\end{definition}

For applications to Cremona groups we are mostly interested in classification of \emph{rational} del Pezzo varieties. There is the result of I. Cheltsov, V. Przyjalkowski and C. Shramov which states that almost all nodal non-$\QQ$-factorial quartic double solids are rational with one family of exceptions (see~\cite[Theorem 1.5]{CPS3}). Also they conjectured that nodal $\QQ$-factorial quartic double solids are not rational (see~\cite[Conjecture 1.9]{CPS3}). This is the reason why we consider only non-$\QQ$-factorial quartic double solids here.

In this paper we are interested in the following problem: classify nodal non-$\QQ$-fac\-torial rational $G$-del Pezzo threefolds of degree 2 that have no $G$-equivariant biratio\-nal map to a ``more simple'' $G$-Fano threefold (for example $\mathbb{P}^{3}$ or a quadric in $\mathbb{P}^{4}$) or to a $G$-Mori fibration with the base of positive dimension. In this paper we give a partial answer to this question.

The main result of this paper is the following theorem:
\begin{theorem}\label{th1} Let $X$ be a nodal non-$\QQ$-factorial del Pezzo threefold of degree $2$ and $G$ be a finite subgroup of $\Aut(X)$ such that $X$ is a rational $G$-Fano variety. Suppose that $X$ is $G$-birationally rigid. Then we have the following possibilities:
\begin{enumerate}
\item
$X$ has $16$ nodes. There is a 3-dimensional family of such varieties. All of them are $\Aut(X)$-birationally superrigid (see~\cite{Che}).
\item
$X$ has $15$ nodes. The variety of such type is unique and described in Propositi\-on~\ref{pr2}.
\item
$X$ has $12$ nodes. There are 3 types and 6 families of such varieties (two $2$-dimensional and four $1$-dimensional families). They are described in Propo\-sitions~\ref{pr5}, \ref{pr6}, \ref{pr7}, \ref{pr8}, \ref{pr10} and \ref{pr11}.
\item
$X$ has $10$ nodes. There is a $1$-dimensional family of such varieties. They are described in Proposition~\ref{pr12}.
\item
$X$ has $8$ nodes. There is a $2$-dimensional family of varieties and one distinguis\-hed variety. They are described in Propositions~\ref{pr13} and \ref{pr14}.
\end{enumerate}
\end{theorem}
\begin{remark} This theorem gives only necessary conditions for the variety $X$ except for the case of $15$ nodes. We mostly don't know if the variety $X$ is $G$-biratio\-nally rigid in other cases.
\end{remark}

\textbf{Acknowledgements.} The author thanks Yuri Prokhorov, Constantin Shramov and Andrey Trepalin for useful discussions. This work is supported by Russian Science Foundation under grant 18-11-00121.

\section{Geometry of nodal quartic double solids}

Throughout this paper $X$ is a nodal rational non-$\QQ$-factorial del Pezzo threefold of degree 2. It is well-known that in this case we have the half-anticanonical double cover $\pi:X\to \PP^{3}$ with ramification at the nodal quartic $Q\subset \PP^{3}$ (note that nodes of the quartic $Q$ are images of nodes of the variety $X$). Let us denote by $\theta:X\to X$ the Geiser involution which naturally arises from the double cover.

Also we use the following notation:
\begin{itemize}
\item
$G\subset \Aut(X)$ is a finite subgroup such that $X$ is $G$-minimal and $G$-birationally rigid;
\item
$\widetilde{G}$ and $\widetilde{\Aut}(X)$ are  corresponding subgroups in $\operatorname{PGL}_{4}(\operatorname{k})$;
\item
$\Cyc_{n}$ is a cyclic group of order $n$;
\item
$\Dih_{2n}$ is a dihedral group of order $2n$;
\item
$\Sym_{n}$ is a symmetric group of degree $n$;
\item
$\Alt_{n}$ is an alternating group of degree $n$.
\end{itemize}

$G$-del Pezzo threefolds were partially classified by Yu. Prokhorov in the paper~\cite{Pro1} and the next theorem is a part of its main result:
\begin{theorem}
Let $X$ be a $G$-del Pezzo threefold of degree $2$. There are the following possibilities:\newline
\begin{center}
\begin{tabular}{|c|c|c|c|c|c|c|c|c|}
\hline
type & $r$ & $X$ & $\widehat{X}$ & $Z$ & $\Delta'$ & $\Delta''$ & $p$ & $s$\tabularnewline
\hline
$15^{\circ}$ &$1$ & $V_{2}$ &$-$ & $\operatorname{pt}$ & $E_7$ & $-$ & $0$ & $10-h$\tabularnewline

$16^{\circ}$ &$2$ & (5.2.7) &$-$& $\PP^{1}$ & $D_6$ & $A_{1}$ & $0$ & $11-h$ \tabularnewline

$17^{\circ}$ &$2$ & (5.2.2) &$-$& $\PP^{2}$ & $A_6$ & $-$ & $0$ & $11$\tabularnewline

$18^{\circ}$ &$2$ &(5.2.13)& $V_{3}$ & $\operatorname{pt}$ & $E_6$ & $-$ & $2$ & $11-h$, $h\le 5$ \tabularnewline

$19^{\circ}$ &$3$ & 4.2.1 &$-$& $(\PP^1)^2$ & $A_5$ & $A_2$ & $0$ & $12$\tabularnewline

$20^{\circ}$ &$3$ & &(5.2.3)& $\PP^{2}$ & $A_5$ & $A_{1}$ & $2$ & $12$\tabularnewline

$21^{\circ}$ &$3$ & &$V_{4}$ & $\operatorname{pt}$ & $D_5$ & $A_{1}$ & $4$ & $12-h$, $h\le 2$\tabularnewline

$22^{\circ}$ &$4$ & &(5.2.8) &$\PP^{1}$ & $D_4$ & $3A_1$ & $8$ & $13-h$ \tabularnewline

$23^{\circ}$ &$4$ && $V_{5}$ & $\operatorname{pt}$ & $A_4$ & $A_2$ & $6$ & $13$ \tabularnewline

$24^{\circ}$ &$5$ && (5.2.5) & $\PP^{2}$ & $A_3$ & $A_1\times A_3$ & $12$ & $14$\tabularnewline

$25^{\circ}$ &$6$ & 8.1 & $V_{6}$& $\PP^{2}$ & $A_2$ & $D_5$ & $20$ & $15$\tabularnewline

$26^{\circ}$ &$7$ & 7.7 & $\PP^{3}$ & $\operatorname{pt}$ & $A_{1}$ & $D_{6}$ & $32$ & $16$\tabularnewline

\hline
\end{tabular}
\end{center}

Here $r$ is the rank of the divisor class group, $\widehat{X}/Z$ is a primitive birational model of the variety $X$ (see~\cite[Theorem 3.9]{Pro1}), $\Delta'$ and $\Delta''$ are root systems canonically associated with $X$, $p$ is the number of planes on $X$ (see Definition 2.5 below), $s$ is the number of singularities if all of them are nodes and $h:=h^{1,2}(\widehat X)$, where $\widehat X$ is the standard resolution of $X$ (for details see \cite[Introduction]{Pro1}).
\end{theorem}
We will consider all types of varieties in this table separately.

We are looking for $G$-birationally rigid varieties, and the following easy but very useful lemma helps us to exclude many possibilities:
\begin{lemma}\label{le1} Assume that the variety $X$ is $G$-birationally rigid. Then the $G$-orbit of a singular point cannot be of length 1, 2, 3 or 5. Also there are no $\widetilde{G}$-invariant lines in $\mathbb{P}^{3}\supset Q$.
\end{lemma}
\begin{proof}
Let $p$ be a singular point of $X$. If $p$ is a $G$-invariant singular point then the projection from the image of this point in $\mathbb{P}^{3}$ gives us an equivariant link from $X$ to a conic fibration. If the $G$-orbit of $p$ consists of two points, then the projection from the line passing through their images gives us an equivariant link with a fibration by rational surfaces that can be transformed into a Mori fibration with the base of positive dimension. The same is true if we have a $\widetilde{G}$-invariant line in $\mathbb{P}^{3}$.

If the $G$-orbit of $p$ consists of three points, then we have several cases. If their images in $\mathbb{P}^{3}$ are collinear, then we again can apply the projection from the line passing through them. If their images are not collinear, but the lines passing through pairs of this points lie on the quartic $Q$, then the intersection of $Q$ and the plane spanned by the orbit of $p$ consists of four lines, so we have a $\widetilde{G}$-invariant line and again, thus, can apply the projection from it. If the lines passing through pairs of this points don't lie on the quartic $Q$, then $X$ has a link to another Mori fiber space by~\cite[Proposition 3.4]{Avi3}.

If the $G$-orbit of $p$ consists of five points, then these points are in general position. Indeed, if they are lying on one plane, then the intersection of the quartic $Q$ with this plane is a quartic curve with 5 singular points of the same type. Then either this curve is a union of four lines in general position, either it is a double quadric. In the first case the sixth singular point is a $G$-invariant point, thus the group $G$ cannot act transitively on five singular points (moreover, one of them is $G$-invariant). In the second case there is a sixth $G$-invariant singular point of $Q$ lying on our double conic (see Proposition~\ref{pr1} below), so we can use the projection from this point. So these points are in general position and we can consider the family of twisted cubics passing through them. It is well-known that there is unique twisted cubic passing through six points in general position. Since $C\cdot Q=12$, general twisted cubic intersects the quartic $Q$ at five nodes and two additional points. Thus the preimage of a general twisted cubic is a rational curve and we have a structure of $G$-equivariant fibration by rational curves on $X$ and it is not $G$-birationally rigid.
\end{proof}
\begin{corollary} The group $\widetilde{\Aut}(X)$ cannot contain a cyclic subgroup of index $1$ or $2$. In particular, it cannot be cyclic or dihedral.
\end{corollary}
\begin{proof} If $g$ is an element of finite order in $\operatorname{PGL}_{4}(\operatorname{k})$ then there is at least four $g$-fixed points in $\PP^{3}$ in general position. In any case we can find a couple of $\widetilde{G}$-fixed points or a $\widetilde{G}$-orbit of length 2, so we have a $\widetilde{G}$-invariant line for any group $G$.
\end{proof}

\begin{remark}\label{re1} Assume that $[p_{1}, p_{2}, p_{3}, p_{4}]$ is a $G$-orbit of singular points in general position. We may assume that their images in $\mathbb{P}^{3}$ have coordinates $(1:0:0:0), (0:1:0:0), (0:0:1:0)$ and $(0:0:0:1)$. Then the equation of $Q$ is of degree at most $2$ in any variable. Thus we have the birational map $\PP(2, 1, 1, 1, 1)\dasharrow \PP(2, 1, 1, 1, 1)$, that can be explicitly described by the following formula:
$$
(y:x_{0}:x_{1}:x_{2}:x_{3})\longrightarrow (yx_{0}x_{1}x_{2}x_{3}:x_{1}x_{2}x_{3}:x_{0}x_{2}x_{3}:x_{0}x_{1}x_{3}:x_{0}x_{1}x_{2}).
$$
One can easily check that in general such map transforms quartic double solid to another quartic double solid. Thus, $G$-birational rigidity of $X$ gives us a strong restrictions on coefficients of the equation of $Q$. But if the equation of $X$ is symmetric with respect to $x_{0}, x_{1}, x_{2}, x_{3}$ then automatically this transformation is a birational automorphism of $X$.
\end{remark}
 \begin{definition}By a plane on a del Pezzo variety $X$ we mean an irreducible surface $\Pi\subset X$ such that $(-\frac{1}{2}K_{X})^{2}\cdot\Pi=1$.
 \end{definition}
 \begin{propos}\label{pr1} Let $H\subset\PP^{3}$ be a plane such that $H\cap Q$ is a double conic (we will call such planes tropes). Then the preimage $\pi^{-1}(H)$ is a couple of planes on $X$, and every plane on $X$ can be obtained in this way. Every plane on $X$ contains exactly $6$ nodes, and every couple of planes contain exactly two common nodes.
 \end{propos}
 \begin{proof} The first statement is obvious. The second statement follows easily from a direct computation in coordinates. Indeed, we may assume that $H$ is given by the equation $x_{0}=0$ and $Q$ has an equation of the form $$F_{2}(x_{1}, x_{2}, x_{3})^{2}+x_{0}G_{3}(x_{1}, x_{2}, x_{3})+x_{0}^{2}G_{2}(x_{1}, x_{2}, x_{3})+x_{0}^{3}G_{1}(x_{1}, x_{2}, x_{3})+ax_{0}^{4}=0.$$
 Singularities on $H$ are given by $x_{0}=F_{2}(x_{1}, x_{2}, x_{3})=G_{3}(x_{1}, x_{2}, x_{3})=0$ and there are exactly 6 of them (if the intersection of corresponding surfaces is not transversal at some point then this point is more complicated singularity than node).
 \end{proof}

 \begin{corollary} The variety $X$ of type $21^{\circ}$ of Theorem 2.1 is never $G$-biratio\-nally rigid.
\end{corollary}
\begin{proof}
In this case we have exactly $2$ tropes. The intersection of these tropes is an $\widetilde{\Aut}(X)$-invariant line.
\end{proof}

All other types we will study separately in the following sections.

\section{Case $26^{\circ}$}

 Varieties of this type was recently studied by I.~Cheltsov. In the paper~\cite{Che} he proved that all such varieties are $\Aut(X)$-birationally superrigid. The group $\widetilde{\Aut}(X)$ is isomorphic to $\Cyc_{2}^{4}\rtimes H$ where $H$ is a subgroup of $\Sym_{6}$ and $\operatorname{PGL}_{2}(\operatorname{k})$ which preserves some $6$-tuple of points on $\PP^{1}$ (see~\cite[Lemma 10]{Che}). More precisely, he proved the following theorem:
 \begin{theorem}~\cite[Theorem 17]{Che} Let $X$ be a quartic double solid with $16$ nodes and $G\subset\Aut(X)$ is a subgroup such that $\operatorname{Cl}^{G}(X)\simeq \mathbb{Z}$ and $\Cyc_{2}^{4}\subset \widetilde{G}$. Then $X$ is $G$-birationally superrigid.
 \end{theorem}

 \begin{remark} The condition $\Cyc_{2}^{4}\subset \widetilde{G}$ is not necessary, so the question of classificati\-on of all subgroups $G\subset\Aut(X)$ such that the variety $X$ is $G$-birationally rigid is still open.
 \end{remark}

\section{Case $25^{\circ}$}

This case was completely studied in the previous author's paper~\cite{Avi3}. The main result of~\cite{Avi3} is the following
\begin{propos}\label{pr2} The variety $X$ is $G$-birationally rigid only in the following situation: it can be given by the equation $$y^{2}-4\sum\limits_{i=0}^{4}x_{i}^{4}+(\sum\limits_{i=0}^{4}x_{i}^{2})^{2}=\sum\limits_{i=0}^{4}x_{i}=0$$ in $\PP(2, 1, 1, 1, 1, 1)$ and $G$ is isomorphic to $\Sym_{5}\times \Cyc_{2}$, $\Alt_{5}\times \Cyc_{2}$ or $\Sym_{5}$ (twisted subgroup in $\Sym_{5}\times \Cyc_{2}$ which is not coincide with the first factor). Moreover, in this case $X$ is $G$-birationally superrigid.
\end{propos}

\section{Case $24^{\circ}$} In this case $X$ has a small $\QQ$-factorialization $\widetilde{X}$ which is the blow-up of a singular del Pezzo threefold $Y$ of degree $5$ at three points $p_{1}, p_{2}$ and $p_{3}$. The variety $Y$ has divisor class group of rank 2. Moreover, $Y$ has exactly one node and fits into the following diagram
\[
\xymatrix{
&\widehat{Y}\ar[dl]_{f}\ar[dr]_{\xi}\ar@{-->}[rr]
&&\widehat{Y}^+\ar[dl]^{\xi^{+}}\ar[dr]^{f^+}
\\
\PP^{2}=Z&&Y&&Z^+=\PP^{1}
}
\]
where $\xi$, $\xi^+$ are small $\QQ$-factorializations, $\widehat Y \dashrightarrow \widehat Y^+$ is a flop, $f$ is a $\PP^{1}$-bundle and $f^+$ is a quadric bundle (see~\cite[Theorem 5.2]{Pro1}; see also~\cite[Theorems 3.5 and 3.6]{JaP} for detailed description of $\widehat Y$ and $\widehat Y^+$).

The variety $X$ has $14$ nodes and contains $12$ planes, thus the quartic $Q$ also has $14$ nodes and $6$ tropes.
\begin{lemma} The group $\widetilde{\Aut}(X)$ acts on the configuration of singular points and tropes and this action is faithful.
\end{lemma}
\begin{proof} Obviously the group $\widetilde{\Aut}(X)$ permutes singularities and tropes and preserves the relation "a singularity lies on a trope". If an element $g\in\widetilde{\Aut}(X)$ preserves all singularities then there is a plane in $\PP^{3}$ containing $13$ singular points or a line containing at least $7$ singular points (because the fixed locus of an element in $\operatorname{PGL}_{4}(\operatorname{k})$ is the projectivization of eigenspaces of a corresponding element in $\operatorname{GL}_{4}(\operatorname{k})$), which is impossible.
\end{proof}

Let $\sigma:\widetilde{Y}\to Y$ be the blow-up of the unique singular point. Then $\widetilde{Y}$ is a smooth Fano threefold of Picard rank $3$. Indeed, since $-K_{\widetilde{Y}}=-\sigma^{*}(K_{Y})-E$ the anticanonical linear system is the proper transform of the system of quadrics in $\PP^{6}$ passing through the singular point (here we consider the half-anticanonical embedding $Y\subset\PP^{6}$). Obviously, such a linear system separates points and tangent vectors at the singular point, so it is very ample.

One can easily calculate that $(-K_{\widetilde{Y}})^{3}=38$. According to~\cite{MoM} there are three families of varieties with required properties: the blow-up of a smooth quartic at two points in general position; the blow-up of a quartic at two skew lines; the blow-up of $\PP^{1}\times\PP^{2}$ at the smooth curve $C$ of bidegree $(2, 1)$. In the first two cases there is no surface $S\subset \widetilde{Y}$ that is isomorphic to $\PP^{1}\times\PP^{1}$ and $\mathcal{O}_{S}(S)\simeq \mathcal{O}(-1, -1)$. In the last case $S$ is a proper transform of the unique divisor $D$ of bidegree $(0, 1)$ which contains $C$. Let $E$ be the exceptional divisor of the blow-up of $C$. Then $p_{i}$'s don't lie on the proper transform of $E$, since otherwise the line on $E$ passing through $p_{i}$ has negative intersection with $-K_{\widetilde{Y}}$ and this is impossible. So we have the following sequence of morphisms
$$X\longleftarrow \widetilde{X}\longleftarrow \widetilde{Y}'\longrightarrow \PP^{1}\times\PP^{2},$$ where the first morphism is the contraction of thirteen curves with zero intersection with canonical class, the second morphism is the contraction of the proper transform of the divisor $D$ and the last morphism is the blow-up of the curve $C$ and three points $p_{i}'$.

\begin{lemma} Singular points on $X$ are images of proper transforms of the follo\-wing subvarieties in $\PP^{1}\times\PP^{2}$:
\begin{itemize}
\item
the divisor on $\PP^{1}\times\PP^{2}$ of bidegree $(0, 1)$ that contains $C$;
\item
three curves of bidegree $(1, 0)$ passing through $p_{i}'$ for some $i$;
\item
six curves of bidegree $(0, 1)$ passing through $p_{i}'$ for some $i$ that intersect $C$;
\item
three curves of bidegree $(1, 1)$ passing through $p_{i}'$ and $p_{j}'$ for some different $i$ and $j$ that intersect $C$;
\item
a curve of bidegree $(1, 2)$ passing through $p_{1}'$, $p_{2}'$ and $p_{3}'$ that intersect $C$.
\end{itemize}

Planes on $X$ are proper transform of the following surfaces:
\begin{itemize}
\item
three exceptional divisors of the blow-up of points $p_{i}'$;
\item
three divisors of bidegree $(1, 0)$ passing through $p_{i}'$ for some $i$;
\item
six divisors of bidegree $(0, 1)$ passing through $p_{i}'$ and $p_{j}'$ for some $i$ and $j$ which are tangent to $C$.
\end{itemize}
\end{lemma}
\begin{proof} One can easily check that subvarieties described above give us singular points and planes. Since there are exactly 14 nodes and 12 planes, we have described all of them.
\end{proof}
\begin{corollary} Nodes of the quartic $Q$ and tropes form a configuration that can be obtained from the $(15_{4}, 10_{6})$-configuration of nodes and tropes of a quartic with $15$ nodes (see~\cite{Avi3} for detailes) by removing one node and four tropes that contain this node.
\end{corollary}
We denote such a configuration by $(14, 6_{6})$-configuration: every trope contains six nodes, eight nodes lie on three tropes and remaining six nodes lie on two tropes.
\begin{remark} One can visualize this configuration in the following way: tropes are faces of a cube, nodes are vertices of a cube and centers of faces, every trope contains five nodes lying on the corresponding face and the center of the opposite face.
\end{remark}
\begin{lemma}  The group $\widetilde{\Aut}(X)$ is a subgroup of $\Sym_{4}\times \Cyc_{2}$ (the group of symmetries of the cube).
\end{lemma}
\begin{proof} From the previous remark we have the natural homomorphism from $\widetilde{\Aut}(X)$ to the group of symmetries of a cube. Now we need to proof that this homomorphism is injective. Assume that the automorphism $g\in \widetilde{\Aut}(X)$ fixes four nodes corresponding to four vertices of the cube lying on three intersecting edges. Without loss of generality we may assume that these nodes have coordinates $(1:0:0:0), (0:1:0:0), (0:0:1:0)$ and $(0:0:0:1)$. The action of the automorphism $g$ is diagonal in this coordinates: $g(x_{0}:x_{1}:x_{2}:x_{3})=(\alpha_{0}x_{0}:\alpha_{1}x_{1}:\alpha_{2}x_{2}:\alpha_{3}x_{3})$. All tropes are $g$-invariant. Three of them has an equation of the form $x_{i}=0, i=1, 2, 3$, and they are $g$-invariant automatically. Remaining three tropes has equations $a_{10}x_{0}+a_{11}x_{1}+a_{12}x_{2}=0, a_{20}x_{0}+a_{21}x_{1}+a_{23}x_{3}=0$ and $a_{30}x_{0}+a_{32}x_{2}+a_{33}x_{3}=0$ and they are $g$-invariant if $\alpha_{0}=\alpha_{1}=\alpha_{2}$, $\alpha_{0}=\alpha_{1}=\alpha_{3}$ and $\alpha_{0}=\alpha_{2}=\alpha_{3}$ respectively. So all $\alpha$'s are equal and $g$ is trivial automorphism.
\end{proof}
\begin{lemma} The normal subgroup $\Cyc_{2}^{3}\subset \Sym_{4}\times \Cyc_{2}$ acts on the set of tropes with three orbits of length $2$. A node lying on two tropes in one orbit doesn't lie on a third trope. As a consequence, if $X$ is $G$-birationally rigid, then there is an $\widetilde{\Aut}(X)$-invariant triple of skew or coplanar lines in $\PP^{3}$.
\end{lemma}
\begin{proof} First two statements are easy observations. Thus, we have three $\Cyc_{2}^{3}$-inva\-riant lines $l_{1}, l_{2}$ and $l_{3}$ (intersections of two tropes) that form an  $\widetilde{\Aut}(X)$-invariant triple of lines. Since the variety $X$ is $G$-birationally rigid, we have no $\widetilde{\Aut}(X)$-invariant lines by Lemma~\ref{le1}, and there are three possibilities: they are either skew, or coplanar, or intersect at one point.

If these lines intersect at one point then this point is common for all tropes and it is not a node. We consider three tropes which has a common node. These tropes has two common points and thus they have a common line and, consequently, have second common node which is not our case (one can easily deduce this from the Corollary 5.2).
\end{proof}
\begin{lemma} If $X$ is $G$-birationally rigid then lines $l_{j}$ cannot be skew.
\end{lemma}
\begin{proof}
In this case we have an $\widetilde{\Aut}(X)$-invariant quadric such that our lines lie in one family of lines on this quadric. We know that the group $\Cyc_{2}^{3}\cap\widetilde{\Aut}(X)$ preserves every line in our triple, so it preserves every line in the family. The group $\widetilde{\Aut}(X)/(\Cyc_{2}^{3}\cap\widetilde{\Aut}(X))$ is a subgroup of $\Sym_{3}$. If it is a cyclic group then we have two $\widetilde{\Aut}(X)$-invariant lines in the family, but this is impossible by Lemma~\ref{le1}. So $\widetilde{\Aut}(X)/(\Cyc_{2}^{3}\cap\widetilde{\Aut}(X))\simeq \Sym_{3}$ and $\widetilde{\Aut}(X)\simeq\Sym_{4}$ (there are two such subgroups) or $\Sym_{4}\times\Cyc_{2}$. If $\widetilde{\Aut}(X)\simeq\Sym_{4}\times\Cyc_{2}$ then we have an embedding of the group $\Sym_{4}\times\Cyc_{2}$ into $\operatorname{PGL}_{2}(\operatorname{k})$ (it arises from the action of the group $\widetilde{\Aut}(X)$ on the second family of lines on the invariant quadric), but it is impossible. Thus $\widetilde{\Aut}(X)\simeq\Sym_{4}$ and acts in the following way: it acts on $\PP^{1}\times\PP^{1}$ as $\Sym_{3}$ on the first component and as $\Sym_{4}$ on the second component.

The line $l_{j}$ is $\Dih_{8}$-invariant, and there is unique $\Dih_{8}$-invariant pair of points on $l_{j}$. These points must be singular points of the quartic $Q$ (we denote them by $p_{j, 1}$ and $p_{j, 2}$). Also we can easily describe tropes $P_{j, 1}$ and $P_{j, 2}$ which contain $l_{j}$: they are spanned by $l_{j}$ and lines in another family of lines on the quadric passing through $p_{j, 1}$ or $p_{j, 2}$ respectively. Thus we can determine remaining eight nodes: they are intersections $P_{1, k}\cap P_{2, l}\cap P_{3, m}$, where $k, l, m\in\{1, 2\}$.

Now we can explicitly calculate such points. We assume that our quadric is the image of the embedding $\PP^{1}\times\PP^{1}\to \PP^{3}, ((x:y), (p:q))\mapsto (xp:yp:xq:yq).$ Assume that the group $\Sym_{4}$ acts on the second $\PP^{1}$ with generators given by matrices
$$\begin{pmatrix}
    -1       & i \\
    i & -1
\end{pmatrix},
\begin{pmatrix}
    1-i       & -1+i \\
    1+i & 1+i
\end{pmatrix}.$$ Also we assume that $l_{1}$ is the image of $(1:0)\times \PP^{1}$, $l_{2}$ is the image of $(0:1)\times \PP^{1}$ and $l_{3}$ is the image of $(1:1)\times \PP^{1}$ (one can easily calculate how $\Sym_{4}$ acts on the first $\PP^{1}$, but we don't need this calculation). Then we can compute equations of all tropes and coordinates of all singular points. For example, the trope $P_{1, 1}$ has an equation $x_{2}=x_{3}$ and contains singularities with coordinates $(1:1:0:0), (1:-1:0:0), (1+i:0:1:1), (0:1-i:1:1), (1+i:0:1:1), (0:1-i:1:1)$. One can check that these points don't lie on a conic, this is a contradiction.
\end{proof}

Now we assume that lines $l_{j}$ are coplanar. Let $P$ be the plane spanned by the lines $l_{j}$. Since it contains six nodes of the quartic $Q$ and the intersection $P\cap Q$ is not a double conic, this intersection is a union of four lines in general position. Thus we have a natural map from $\widetilde{G}$ to the permutation group of four lines.
\begin{lemma}This map is either injective or its kernel is $\Cyc_{2}$. If $X$ is $G$-birationally rigid then the image of this map is isomorphic to either $\Alt_{4}$ or $\Sym_{4}$. The action of $\widetilde{G}$ on $P$ is the projectivisation of an irreducible 3-dimensional representation.
\end{lemma}
\begin{proof} Obviously, only trivial self-map of the plane $P$ preserves every line in $P\cap Q$. The group $\widetilde{\Aut}(X)$ is a subgroup of $\Sym_{4}\times \Cyc_{2}$ and only $\Cyc_{2}$ can act trivially on the set $P\cap\operatorname{Sing}(Q)$ and thus on $P$ (one can see it checking the action of the automorphism group on $(14, 6_{6})$-configuration). This proves the first statement. If the image of this map is not $\Alt_{4}$ or $\Sym_{4}$ then we have a $\widetilde{G}$-invariant line ($l_{j}$ or a line in $P\cap Q$).

The third statement easily follows from the fact that the group $\operatorname{PGL}_{3}(\operatorname{k})$ acts transitively on set of quadruples of lines in general position and the projectivisation of a 3-dimensional representation of the group $\Alt_{4}$ or $\Sym_{4}$ has the invariant quadruple of lines.
\end{proof}

\begin{lemma} There is a $\widetilde{G}$-invariant point $p$ that doesn't lie on $P$.
\end{lemma}
\begin{proof} Let $\widehat{G}\subset \operatorname{SL}_{4}(\operatorname{k})$ be the preimage of the group $\widetilde{G}\subset \operatorname{PGL}_{4}(\operatorname{k})$. It is a finite group, and the standard representation of $\widehat{G}$ contains the 3-dimensional subrepresentation, corresponding to the plane $P$. Thus it contains also one-dimensi\-onal subrepresentation that corresponds to a point $p$.
\end{proof}

\begin{lemma} $\PP^{3}$ as a $\widetilde{G}$-variety is isomorphic to  the projectivisation of a $4$-di\-mensional representation of the group $\widetilde{G}$.
\end{lemma}
\begin{proof} Let us construct such a representation. Let $\PP^{3}=\PP(V)$, $P=\PP(U)$ and $p=\PP(W)$, where $U, W\subset V$ be the corresponding vector subspaces of dimension $3$ and $1$ respectively. As we noticed above we may consider $U$ as a standard $3$-dimensional representation of $\widetilde{G}$. Then we can define in the unique way an action of $\widetilde{G}$ on $W$ and consequently on $V$ such that it agrees with the action on $\PP^{3}$. Obviously, this action is a representation.
\end{proof}

\begin{lemma}\label{le14} In some coordinates $\Alt_{4}\subset G$ acts by permutation of coordinates.
\end{lemma}
\begin{proof}
The tensor product of the unique three-dimensional irreducible represen\-tation of the group $\Alt_{4}$ with any one-dimensional representation is isomorphic to the original representation. Thus $V$ as the representation of the group $\Alt_{4}$ is the tensor product of the standard four-dimensional representation with some one-dimensional representation. But all such representations have the same projectivizations.
\end{proof}
\begin{lemma} Lines $l_{i}$ cannot be coplanar.
\end{lemma}
\begin{proof}
Assume that $P$ has an equation $x_{0}+x_{1}+x_{2}+x_{3}=0$. From the previous lemma we may assume that one line in $P\cap Q$ has an equation $x_{0}=x_{1}+x_{2}+x_{3}=0$ and the line $l_{1}$ has an equation $x_{0}+x_{1}=x_{2}+x_{3}=0$. Consequently, the trope $P_{1, 1}$ has an equation $x_{0}+x_{1}+\alpha(x_{2}+x_{3})=0, \alpha\neq\pm 1$. Now we can easily compute equations of all other tropes and coordinates of singular points of $Q$. The trope $P_{1, 1}$ contains singular points of the quartic $Q$ with coordinates $(0:0:1:-1), (1:-1:0:0), (-1-2\alpha:1:1:1)$, $(1:1:-\frac{2+\alpha}{\alpha}:1), (1:-1-2\alpha:1:1)$ and $(-\frac{\alpha}{2+\alpha}:-\frac{\alpha}{2+\alpha}:-\frac{\alpha}{2+\alpha}:1)$. One can easily check that the unique curve of degree $2$ passing through them is a union of two lines, a contradiction.
\end{proof}
As a consequence of Lemmas 5.5, 5.6 and 5.11 we obtain the following proposition:
\begin{propos} In the case $24^{\circ}$ the variety $X$ is never $G$-birationally rigid.
\end{propos}
\begin{proof} If the variety $X$ is $G$-birationally rigid then the lines $l_{i}$ cannot intersect at a single point by Lemma 5.6, cannot be skew by Lemma 5.7 and cannot be coplanar the Lemma 5.12. This is a contradiction.
\end{proof}
\section{Case $23^{\circ}$}

In this case $X$ has a small $\QQ$-factorialization $\widetilde{X}$ which is the blow-up of $V_{5}$ in three points $p_{1}$, $p_{2}$ and $p_{3}$, where $V_{5}$ is a smooth del Pezzo threefold of degree $5$. Those points should be in general position, that is, any two of them don't lie on a line on $V_{5}$ and all of them don't lie on a conic on $V_{5}$.

\begin{propos} In this case $X$ is never $G$-birationally rigid.
\end{propos}
\begin{proof} We have exactly three tropes in this case. Each of them contains six nodes and any two of them contain exactly two common nodes. If their unique common point is a node then automatically this node is $\widetilde{G}$-invariant. If their common point is not a node then tropes contain $12$ nodes while there are $13$ nodes on the quartic $Q$, thus remaining node is $\widetilde{G}$-invariant. In any case the variety $X$ is not $G$-birationally rigid.
\end{proof}
\begin{remark} In fact one can prove that tropes contain a common singularity (even if we allow more complicated singularities). Tropes are images of exceptional divisors of the blow-up $\widetilde{X}\to V_{5}$. It is well known that the Hilbert scheme of subschemes $C\subset V_{5}$ with the Hilbert polynomial $3t+1$ is isomorphic to $\operatorname{Gr}(2, 5)$ (see~\cite[Proposition 2.46]{San}) and thus it is six-dimensional. Then one can prove that there is a (possibly reducible) curve in this family passing through $p_{1}, p_{2}$ and $p_{3}$. The map $\widetilde{X}\to X$ contracts the proper transform of this curve to a singular point of the variety $X$, and its image in $\PP^{3}$ lies in every trope.
\end{remark}

\section{Case $22^{\circ}$} In this case we have exactly eight planes on $X$, thus there are exactly four tropes. Let us denote them by $P_{i}, i=0, ..., 3$. Since there are no $\widetilde{G}$-invariant lines, we have two possibilities: either all tropes have one common point and any three of them have no common line, or they are in general position.

\subsection{The case of four tropes with a common point}

In what follows all lemmas without proofs are analogous to lemmas from the case $24^{\circ}$.
\begin{lemma} Assume that all tropes $P_{i}$ have one common point $p$. Then there is a $\widetilde{G}$-invariant plane $P$ that doesn't contain $p$.
\end{lemma}

There is a short exact sequence $$0\longrightarrow G'\longrightarrow \widetilde{G}\longrightarrow G''\longrightarrow 0,$$ where the group $G'$ acts trivially on the plane $P$ and $G''$ acts faithfully on $P$.
\begin{lemma} The group $G'$ is trivial or cyclic of order $2$.
\end{lemma}
\begin{proof} The line $P_{0}\cap P_{1}$ is $G'$-invariant. It contains two $G'$-fixed points ($p$ and $P_{0}\cap P_{1}\cap P$) and two nodes of $Q$. Thus $G'\simeq \Cyc_{2}$ or $G'$ is trivial.
\end{proof}
\begin{lemma} The group $G''$ is isomorphic to $\Alt_{4}$ or $\Sym_{4}$. The action of $G''$ on $P$ is the projectivisation of irreducible 3-dimensional representation.
\end{lemma}

\begin{lemma} $\PP^{3}$ as a $\widetilde{G}$-variety is isomorphic to the projectivisation of a $4$-dimen\-sional representation of the group $\widetilde{G}$.
\end{lemma}

\begin{corollary} $G$ is isomorphic to either $\Alt_{4}$, $\Sym_{4}$, $\Alt_{4}\times \Cyc_{2}$, $\Sym_{4}\times \Cyc_{2}$ or $\Alt_{4}\rtimes \Cyc_{4}$.
\end{corollary}
\begin{proof} If $G'\simeq \Cyc_{2}$ then $G$ is a central non-stem  extension of $G''$ by $\Cyc_{2}$, since otherwise $G'$ acts trivially on $V$ (recall, a stem extension is a group extension where the base normal subgroup is contained in both the center and the derived subgroup of the whole group).
\end{proof}

\begin{lemma} In some coordinates $\Alt_{4}\subset G$ acts by permutations of coordinates.
\end{lemma}

\begin{lemma} In some coordinates the quartic $Q$ has symmetric equation with respect to the group $\Sym_{4}$ acting by permutations of coordinates.
\end{lemma}
\begin{proof} The equation of the quartic generates one-dimensional subrepresentation of the group $\Alt_{4}$ (which acts by permutations of coordinates) in the space of quartic forms. Thus it is either symmetric or of the form
\begin{equation}
\begin{gathered}
A(x_{0}^{3}x_{1}+x_{1}^{3}x_{0}+x_{2}^{3}x_{3}+x_{3}^{3}x_{2}+\xi(x_{0}^{3}x_{2}+x_{2}^{3}x_{0}+x_{1}^{3}x_{3}+x_{3}^{3}x_{1})+\xi^{2}(x_{0}^{3}x_{3}+x_{3}^{3}x_{0}+\\
x_{1}^{3}x_{2}+x_{2}^{3}x_{1}))+B(x_{0}^{2}x_{1}x_{2}+x_{1}^{2}x_{0}x_{3}+x_{2}^{2}x_{0}x_{3}+x_{3}^{2}x_{1}x_{2}+\xi(x_{0}^{2}x_{2}x_{3}+x_{1}^{2}x_{2}x_{3}+\\
x_{2}^{2}x_{0}x_{1}+x_{3}^{2}x_{0}x_{1})+\xi^{2}(x_{0}^{2}x_{1}x_{3}+x_{1}^{2}x_{0}x_{2}+x_{2}^{2}x_{1}x_{3}+x_{3}^{2}x_{0}x_{2}))+C(x_{0}^{2}x_{1}^{2}+x_{2}^{2}x_{3}^{2}+\\
\xi(x_{0}^{2}x_{2}^{2}+x_{1}^{2}x_{3}^{2})+\xi^2(x_{0}^{2}x_{3}^{2}+x_{1}^{2}x_{2}^{2})).
\end{gathered}
\end{equation}
where $\xi$ is a primitive cubic root of unity. We know that the intersection of $Q$ with $P_{0}$ is a double conic. Since the plane $P_{0}$ is $\Cyc_{3}$-invariant and passes through the point $(1:1:1:1)$, we may assume that it has an equation of the form $x_{0}=\frac{1}{3}(x_{1}+x_{2}+x_{3})$, $x_{1}+\xi x_{2}+\xi^{2} x_{3}=0$ or $x_{1}+\xi^2 x_{2}+\xi x_{3}=0$. One can check that in the second case the equation of $Q\cap P_{0}$ has the form $(x_{3}-x_{2})^{2}F_{A, B, C}(x_{0}, x_{2}, x_{3})$ and in the third case it has the form $(x_{3}-x_{2})G_{A, B, C}(x_{0}, x_{2}, x_{3})$, where $F_{A, B, C}$ and $G_{A, B, C}$ are some quadric and cubic forms respectively depending on $A, B$ and $C$. Thus in these cases $Q\cap P_{0}$ is not a double conic.

In the first case we have $Q(\frac{1}{3}(x_{1}+x_{2}+x_{3}), x_{1}, x_{2}, x_{3})=R(x_{1}, x_{2}, x_{3})^2$. Obviously, $R$ is semi-invariant polynomial with respect to $\Cyc_{3}\subset \Alt_{4}$, so we may assume that
$$R(x_{1}, x_{2}, x_{3})=D(x_{1}^2+\xi^2 x_{2}^{2}+\xi x_{3}^2)+E(x_{1}x_{2}+\xi^2 x_{2}x_{3}+\xi x_{1}x_{3}).$$
Writing down equations on coefficients $A, B, C, D$ and $E$ which we obtain from the equality $Q(\frac{1}{3}(x_{1}+x_{2}+x_{3}), x_{1}, x_{2}, x_{3})=R(x_{1}, x_{2}, x_{3})^2$, we see that $C=\frac{1}{2}(5A+\xi B)$ and either $A=0$ or $B=2A$. In the first case $R=D(x_{1}+\xi x_{2}+\xi^2 x_{3})^2$ and in the second case $R=D(x_{1}+x_{2}+x_{3})(x_{1}+\xi^2 x_{2}+\xi x_{3})$, this is impossible.
\end{proof}
\begin{propos}\label{pr5} In some coordinates the quartic $Q$ has the equation
$$a\left(\sum\limits_{i}x_{i}^{4}\right)+b\left(\sum\limits_{i\neq j}x_{i}^{3}x_{j}\right)+c\left(\sum\limits_{i\neq j}x_{i}^{2}x_{j}^{2}\right)+d\left(\sum\limits_{i\neq j\neq k\neq i}x_{i}^{2}x_{j}x_{k}\right)+ex_{0}x_{1}x_{2}x_{3}=0,$$ where $e=-5b-7d, c=\frac{1}{36}(-59b-37d), a=\frac{1}{72}(-29b+5d)$.
\end{propos}
\begin{proof} We know that the equation of the plane $P_{0}$ is $x_{0}=\frac{1}{3}(x_{1}+x_{2}+x_{3})$. We know that the quartic $Q$ is singular at some point on the line $P_{0}\cap P_{1}$. Note that we may apply the transformation
$$(x_0:x_1:x_2:x_3)\longmapsto\left(x_0+\alpha\sum\limits_{i=0}^{3}x_{i}:x_1+\alpha\sum\limits_{i=0}^{3}x_{i}:x_2+\alpha\sum\limits_{i=0}^{3}x_{i}:x_3+\alpha\sum\limits_{i=0}^{3}x_{i}\right),$$ which is simply a choosing of another basic vector in $W$. Such maps acts transitively on the line $P_{0}\cap P_{1}$ without points $p$ and $P_{0}\cap P_{1}\cap P$, so we may assume without loss of generality that the quartic $Q$ is singular at the point $(\frac{1}{2}:\frac{1}{2}:0:1)$. Now we can solve the corresponding system of equation and obtain the required conditions for coefficients. Then using the direct computation one can check that for these coefficients $Q\cap P_{0}$ is a double conic.
\end{proof}
This is the only variety which can be $G$-birationally rigid in this case.
\subsection{The case of four planes in general position}

We may assume that $P_{i}$ is given by the equation $x_{i}=0$.
\begin{lemma} The points $(1:0:0:0), (0:1:0:0), (0:0:1:0)$ and $(0:0:0:1)$ don't lie on the quartic $Q$.
\end{lemma}
\begin{proof} Assume that $(1:0:0:0)\in Q$. Recall that every two tropes $T_{1}$ and $T_{2}$ have two common singular points and these points are precisely $Q\cap T_{1}\cap T_{2}$, so the point $(1:0:0:0)$ is singular point of $Q$. Since the orbit of a singular point has length at least $4$, points $(0:1:0:0), (0:0:1:0)$ and $(0:0:0:1)$ are also nodes. Every trope contains three other singularities (since every trope contains exactly six nodes). This gives us $16$ singular points in total. Since $13$ is the maximal possible value in this case, we get a contradiction.
\end{proof}
\begin{lemma} The equation of $Q$ in some coordinates has the form
$$\left(x_{0}^{2}+x_{1}^{2}+x_{2}^{2}+x_{3}^{2}+\sum\limits_{0\leq i<j\leq 3}a_{ij}x_{i}x_{j}\right)^{2}+ax_{0}x_{1}x_{2}x_{3}=0,$$
where $a\neq 0$.
\end{lemma}
\begin{proof} We know that the equation of the quartic $Q$ restricted to $P_{0}$ is a complete square and points $(1:0:0:0), (0:1:0:0), (0:0:1:0)$ and $(0:0:0:1)$ don't lie on $Q$. Without loss of generality we may assume that coefficients of monomials $x_{i}^{4}$ are equal to $1$ (otherwise we can apply a diagonal change of coordinates). Thus we may assume that

$$Q(x_{0}, x_{1}, x_{2}, x_{3})=(x_{1}^{2}+x_{2}^{2}+x_{3}^{2}+a_{12} x_{1}x_{2}+a_{13} x_{1}x_{3}+a_{23}x_{2}x_{3})^{2}+$$
$$x_{0}F_{1}(x_{0}, x_{1}, x_{2}, x_{3}).$$

Restricting to $P_{1}$ we obtain $$Q(x_{0}, x_{1}, x_{2}, x_{3})=(x_{0}^{2}+x_{2}^{2}+x_{3}^{2}+a_{02}x_{0}x_{2}+a_{03}x_{0}x_{3}+b_{23}x_{2}x_{3})^{2}+$$
$$+x_{1}F_{2}(x_{0},x_{1}, x_{2}, x_{3}).$$
We see that $(x_{2}^{2}+x_{3}^{2}+a_{23} x_{2}x_{3})^{2}=Q|_{P_{0}\cap P_{1}}=(x_{2}^{2}+x_{3}^{2}+b_{23} x_{2}x_{3})^{2}$, so $b_{23}=a_{23}$ and $$Q(x_{0}, x_{1}, x_{2}, x_{3})=(x_{0}^{2}+x_{1}^{2}+x_{2}^{2}+x_{3}^{2}+a_{02}x_{0}x_{2}+a_{03}x_{0}x_{3}+a_{12}x_{1}x_{2}+a_{13}x_{1}x_{3}+$$
$$+a_{23}x_{2}x_{3})^{2}+x_{0}x_{1}F_{3}(x_{0}, x_{1}, x_{2}, x_{3}).$$
Restricting to $P_{2}$ we obtain
$$Q(x_{0}, x_{1}, x_{2}, x_{3})=(x_{0}^{2}+x_{1}^{2}+x_{3}^{2}+a_{01}x_{0}x_{1}+b_{03}x_{0}x_{3}+b_{13}x_{1}x_{3})^{2}+$$
$$+x_{2}F_{4}(x_{0}, x_{1}, x_{2}, x_{3}).$$
From $Q|_{P_{0}\cap P_{2}}$ we see that $b_{13}=a_{13}$ and from $Q|_{P_{1}\cap P_{2}}$ we see that $b_{03}=a_{03}$. Thus
$$Q(x_{0}, x_{1}, x_{2}, x_{3})=(x_{0}^{2}+x_{1}^{2}+x_{2}^{2}+x_{3}^{2}+a_{01}x_{0}x_{1}+a_{02}x_{0}x_{2}+a_{03}x_{0}x_{3}+a_{12}x_{1}x_{2}+$$
$$+a_{13}x_{1}x_{3}+a_{23}x_{2}x_{3})^{2}+x_{0}x_{1}x_{2}F_{5}(x_{0}, x_{1}, x_{2}, x_{3}).$$
Since $Q|_{P_{3}}$ is a complete square we easily see that $F_{5}(x_{0}, x_{1}, x_{2}, x_{3})=ax_{3}$, and we obtain the required formula.

If $a=0$ then $Q$ is a double quadric which is impossible.
\end{proof}
\begin{corollary} The group $\widetilde{\Aut}(X)$ fits into the natural short exact sequence $$0\longrightarrow G'\longrightarrow \widetilde{\Aut}(X)\longrightarrow G''\longrightarrow 0$$ where $G'$ is a subgroup of $\Cyc_{2}^{2}$ acting by changing signs of even number of coordinates and $G''$ is a subgroup of $\Sym_{4}$.
\end{corollary}
\begin{propos}\label{pr6} If $G'\simeq \Cyc_{2}^{2}$ then all $a_{ij}=0$ and $\Aut(X)\simeq \Cyc_{2}^{2}\rtimes \Sym_{4}\times \Cyc_{2}$. If $G'\simeq \Cyc_{2}$ then $X$ is not $G$-birationally rigid. If $G'$ is trivial then $X$ can be birationally rigid only if in some coordinates all $a_{ij}$ are equal and $\Aut(X)\simeq \Sym_{4}\times \Cyc_{2}$.
\end{propos}
\begin{proof} The first fact is obvious. Assume that $G'\simeq \Cyc_{2}$. Then we may assume without loss of generality that $G'$ generated by an element that changes signs of variables $x_{2}$ and $x_{3}$. So all $a_{ij}=0$ except $a_{01}$ and $a_{23}$. If $a_{01}=\pm a_{23}$ then $\widetilde{G}\simeq \Dih_{8}\times \Cyc_{2}$, otherwise $\widetilde{G}\simeq \Cyc_{2}^{3}$. In any case we have a $\widetilde{G}$-invariant line $x_{0}=x_{1}, x_{2}=x_{3}$, so $X$ cannot be $G$-birationally rigid. If $G'$ is trivial then $\widetilde{G}$ is a subgroup of $\Sym_{4}$ and $X$ can be $G$-birationally rigid only if $\widetilde{G}$ contains $\Alt_{4}$. One can easily check that this is possible only if in some coordinates all $a_{ij}$ are equal. The fact that $\Aut(X)$ is the direct product of $\widetilde{\Aut}(X)$ and $\Cyc_{2}$ is obvious.
\end{proof}

Thus we have a $2$-parametric family of varieties which can be $G$-birationally rigid.

\section{Case $20^{\circ}$}

In this case $X$ has a small $\QQ$-factorialization $\widetilde{X}$ which is the blow-up of a nodal cubic threefold $Y$ of Picard rank $2$ in a general smooth point $p$. Such cubic has exactly $6$ nodes and two structures of $\PP^{1}$-bundle. By Lemma~\ref{le1} the corresponding singular points of $X$ form a single $G$-orbit (let us denote them by $p_1, p_2, p_3, p_{4}, p_{5}, p_{6}$). Remaining singular points of the variety $X$ come from six lines passing through the point $p$. Their images lie on the unique trope and form a single $\widetilde{G}$-orbit lying on a conic. Denote these trope and conic by $P$ and $C$ respectively.

\begin{lemma} There is a $\widetilde{G}$-invariant point $q$ in $\PP^{3}$.
\end{lemma}
\begin{proof} Let $\widehat{G}\subset \operatorname{SL}_{4}(\operatorname{k})$ be the preimage of the group $\widetilde{G}\subset \operatorname{PGL}_{4}(\operatorname{k})$. It is a finite group, and the standard representation of $\widehat{G}$ contains the 3-dimensional subrepresentation, corresponding to the plane $P$. Thus it contains also 1-dimensional subrepresentation that corresponds to a point $q\notin P$.
\end{proof}

There is a short exact sequence $$0\longrightarrow G'\longrightarrow \widetilde{G}\longrightarrow G''\longrightarrow 0,$$ where the group $G'$ acts trivially on $P$ and $G''$ acts faithfully on $P$.
\begin{lemma} The group $G'$ is either trivial or cyclic of order $2$.
\end{lemma}
\begin{proof} If the group $G'$ is not trivial nor cyclic of order $2$ then either all points $p_{i}$ lie on one line passing through $q$ (in this case $G'\simeq \Cyc_{6}$) or they lie on two lines passing through $q$ (and in this case $G'\simeq \Cyc_{3}$). In any case we have a $\widetilde{G}$-invariant line: in the first case this is the line passing through $p_i$'s, in the second case this is the intersection of the plane spanned by $p_i$'s and the plane $P$.
\end{proof}
\begin{lemma} The group $G''$ is isomorphic to $\Alt_{4}$ or $\Sym_{4}$.
\end{lemma}
\begin{proof} The natural map $G''\to\operatorname{PGL}_{2}(\operatorname{k})$ is an embedding since every morphism acting trivially on the curve $C$ acts trivially on the plane $P$. Also we know that the group $G''$ has an orbit of length $6$ on the conic $C$. So $G''$ is isomorphic to $\Cyc_{6}$, $\Sym_{3}$, $\Dih_{12}$, $\Alt_{4}$ or $\Sym_{4}$. But if $G''$ is isomorphic to $\Cyc_{6}$, $\Sym_{3}$ or $\Dih_{12}$ then there is a $G$-invariant pair of points on $C$, so there is a $\widetilde{G}$-invariant line.
\end{proof}

There are unique subgroups isomorphic to $\Alt_{4}$ or $\Sym_{4}$ in $\operatorname{PGL}_{3}(\operatorname{k})$ up to conjugation. They can be constructed in the following way: if $U$ is an irreducible 3-dimensional representation of $G''\simeq\Alt_{4}$ or $\Sym_{4}$ then $\PP(U)$ contains unique $G''$-invariant conic $\widetilde{C}$. The pair $(P, C)$ with an action of the group $G''$ is isomorphic to $(\PP(U), \widetilde{C})$.

\begin{lemma} The projective space $\PP^{3}$ with the action of the group $\widetilde{G}$ is the projec\-tivization of a $4$-dimensional representation of the group $\widetilde{G}$.
\end{lemma}
\begin{proof} Let us construct such a representation. Let $\PP^{3}=\PP(V)$, $P=\PP(U)$ and $q=\PP(W)$, where $U, W\subset V$ are the corresponding vector subspaces of dimension $3$ and $1$ respectively. As we noticed above we may consider $U$ as the standard $3$-dimensional representation of $G''$ and consequently of $\widetilde{G}$. Then we can define in the unique way an action of $\widetilde{G}$ on $W$ and consequently on $V$ such that it agrees with an action on $\PP^{3}$. Obviously, this action is a linear representation.
\end{proof}

\begin{corollary} $\widetilde{G}$ is isomorphic to $\Alt_{4}$, $\Sym_{4}$, $\Alt_{4}\times \Cyc_{2}$, $\Sym_{4}\times \Cyc_{2}$ or $\Alt_{4}\rtimes \Cyc_{4}$.
\end{corollary}
\begin{proof} If $G'\simeq \Cyc_{2}$ then $\widetilde{G}$ is a central non-stem extension of $G''$ by $\Cyc_{2}$, since otherwise $G'$ acts trivially on $V$.
\end{proof}

\begin{lemma} In some coordinates the quartic $Q$ has a symmetric equation with respect to the permutations of coordinates. As a consequence, $\widetilde{\Aut}(X)\simeq \Sym_{4}$ or $\Sym_{4}\times \Cyc_{2}$.
\end{lemma}
\begin{proof} By Lemma~\ref{le14} we may assume that $\Alt_{4}\subset\widetilde{G}$ acts on $V$ by the permuta\-tions of coordinates. If the equation of the quartic $Q$ is not symmetric then it has the form
$$A(x_{0}^{3}x_{1}+x_{1}^{3}x_{0}+x_{2}^{3}x_{3}+x_{3}^{3}x_{2}+\xi(x_{0}^{3}x_{2}+x_{2}^{3}x_{0}+x_{1}^{3}x_{3}+x_{3}^{3}x_{1})+\xi^{2}(x_{0}^{3}x_{3}+x_{3}^{3}x_{0}+x_{1}^{3}x_{2}+$$
$$+x_{2}^{3}x_{1}))+B(x_{0}^{2}x_{1}x_{2}+x_{1}^{2}x_{0}x_{3}+x_{2}^{2}x_{0}x_{3}+x_{3}^{2}x_{1}x_{2}+\xi(x_{0}^{2}x_{2}x_{3}+x_{1}^{2}x_{2}x_{3}+x_{2}^{2}x_{0}x_{1}+$$
$$+x_{3}^{2}x_{0}x_{1})+\xi^{2}(x_{0}^{2}x_{1}x_{3}+x_{1}^{2}x_{0}x_{2}+x_{2}^{2}x_{1}x_{3}+x_{3}^{2}x_{0}x_{2}))+C(x_{0}^{2}x_{1}^{2}+x_{2}^{2}x_{3}^{2}+\xi(x_{0}^{2}x_{2}^{2}+x_{1}^{2}x_{3}^{2})+$$
$$+\xi^2(x_{0}^{2}x_{3}^{2}+x_{1}^{2}x_{2}^{2})),$$
where $\xi$ is a primitive root of the unity of degree $3$. We also know that $Q$ is singular at the $G$-orbit of length $6$, thus $Q$ is singular at the point $(a:a:b:b)$ for some $a\neq\pm b$. Note that we may apply the transformation
$$(x_0:x_1:x_2:x_3)\longmapsto\left(x_0+\alpha\sum\limits_{i=0}^{3}x_{i}:x_1+\alpha\sum\limits_{i=0}^{3}x_{i}:x_2+\alpha\sum\limits_{i=0}^{3}x_{i}:x_3+\alpha\sum\limits_{i=0}^{3}x_{i}\right),$$  that corresponds to choosing another basic vector of the $1$-dimensional subrepresen\-tation $W$. Thus we may assume that the quartic $Q$ is singular at the point $(1:1:0:0)$. Solving the corresponding system of equations we see that $C=-\xi B, A=2\xi B$. But the intersection of such quartic with the plane $P$ is not a double conic $2C$.
\end{proof}

\begin{lemma} The variety $X$ can be given by the equation \begin{equation}\label{eq5}
s_{2}^{2}-s_{1}s_{3}+\frac{1}{4}s_{1}^2s_{2}+A(8s_{1}s_{3}-6s_{1}^2s_{2}+s_{1}^{4})=0
\end{equation}
for some $A\neq 1$, where $s_{j}=\sum\limits_{i=0}^{3} x_{i}^{j}$. The group $\Aut(X)$ is isomorphic to  $\Sym_{4}\times \Cyc_{2}$.
\end{lemma}
\begin{proof} We know that the intersection of the variety $X$ with the hyperplane $\sum\limits_{i=0}^{3} x_{i}=0$ is a double conic. So the equation of the quartic $Q$ has the form
\begin{equation}\label{eq2}
s_{2}^{2}+as_{1}s_{3}+bs_{1}^2s_{2}+cs_{1}^{4}=0,
\end{equation}
where $s_{j}=\sum\limits_{i=0}^{3} x_{i}^{j}$. As in the previous lemma we may assume that the quartic $Q$ is singular at the point $(1:1:0:0)$. Solving the corresponding system of equations we obtain the required formula~\eqref{eq5}.

Assume that $\widetilde{\Aut}(X)\simeq \Sym_{4}\times \Cyc_{2}$. Since $G'$ acts nontrivially on $W$, $\Cyc_{2}$ acts as $-\operatorname{Id}$ on $W$. The generator of $\Cyc_{2}$ acts as
\begin{equation}
\label{eq1}
(x_{0}:x_{1}:x_{2}:x_{3})\longmapsto \left(x_{0}-\frac{1}{2}\sum\limits_{i=0}^{3}x_{i}:x_{1}-\frac{1}{2}\sum\limits_{i=0}^{3}x_{i}:x_{2}-\frac{1}{2}\sum\limits_{i=0}^{3}x_{i}:x_{3}-\frac{1}{2}\sum\limits_{i=0}^{3}x_{i}\right).
\end{equation}
One can easily check that such a map doesn't preserve the quartic $Q$.
\end{proof}

From the previous lemmas we deduce the following proposition:
\begin{propos}\label{pr7} The variety $X$ can be given by the equation ~\eqref{eq5}, the group $\Aut(X)\simeq \Sym_{4}\times \Cyc_{2}$ and $G\simeq \Sym_{4}\times \Cyc_{2}$, $\Sym_{4}$ (which permutes two planes on $X$, otherwise $X$ is not $G$-minimal) or $\Alt_{4}\times \Cyc_{2}$.
\end{propos}

\section{Case $19^{\circ}$}
In this case $X$ is a del Pezzo threefold with $\operatorname{rk}\operatorname{Cl}(X)=3$ that contains no planes.
\begin{theorem}(see~\cite[Theorem 4.2]{Pro1}) There are exactly six complete one-dimen\-sional systems of Weil divisors without fixed components on $X$. Each of them defines a structure of a quadric fibration on some small $\QQ$-factorialization of $X$. They fits into the following graph

\begin{center}
\includegraphics[scale=0.75]{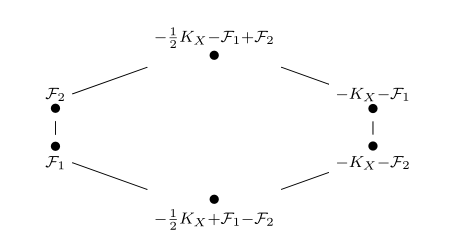}
\end{center}
where two systems are connected by a line iff the corresponding quadric fibrations factors through the same $\PP^{1}$-fibration.
\end{theorem}

Hence we have the following natural short exact sequence
$$0\longrightarrow G'\longrightarrow \Aut(X)\longrightarrow G''\longrightarrow 0$$
where $G'$ preserves all $\mathcal{F}_{i}$'s and $G''$ is a subgroup of $\Dih_{12}$.
\begin{remark} The Geiser involution acts on the hexagon as a central symmetry, because $\mathcal{F}+\tau(\mathcal{F})$ must be a multiple of $-\frac{1}{2}K_{X}$.
\end{remark}

Let $F\in \mathcal{F}_{1}$ be a general member and $\theta(F)$ be the corresponding member of the linear system $-K_{X}-\mathcal{F}_{1}$. Then $F+\theta(F)$ is the preimage of a quadric under the double cover $\pi$. So systems $\mathcal{F}_{1}$ and $-K_{X}-\mathcal{F}_{1}$ give us a system of quadrics parametrized by a conic in the space of quadrics, since a general point $t\in \PP^{3}$ lies on two quadrics in our family (two points in the preimage $\pi^{-1}(t)$ lie in two different members of the linear system $\mathcal{F}_{1}$). We can explicitly write this system in the form
\begin{equation}\label{eq6}p^{2}Q_{1}+pqQ_{2}+q^{2}Q_{3}=0
\end{equation}
where $(p:q)$ parametrize a point in $\PP^{1}$ and $Q_{i}=Q_{i}(x_{0}:x_{1}:x_{2}:x_{3})$ are quadric forms. The equation of the quartic $Q$ in this case is simply the discriminant of~\eqref{eq6} $$Q_{2}^{2}-4Q_{1}Q_{3}=0.$$
\begin{lemma} Three quadrics defined by $Q_{i}=0$ intersect transversally at $8$ points (we will denote quadrics by the same symbol as their equations). These points are nodes of the quartic~$Q$.
\end{lemma}
\begin{proof} Obviously, the common points of $Q_{i}$'s are singular points of $Q$, so there are finite number of them. Let $p$ be one of such points. If $Q_{i}$ is singular at the point $p$ for some $i$ then locally the equation of $Q_{i}$ at this point contains no constant and linear terms, so the equation of the quartic $Q$ cannot contain the quadratic term of maximal rank. The same is true if tangent planes to $Q_{i}$'s are not in general position.
\end{proof}

Now we consider the equation~\eqref{eq6} as a quadric form in variables $x_{i}$ whose coefficients are quadric polynomials in variables $p, q$. Let $M(p:q)$ be the correspon\-ding symmetric matrix. Its determinant is a polynomial of degree $8$ (we denote this polynomial by $D(p:q)$), so in general situation we have $8$ singular members in the family. But in our case the situation is different.

\begin{lemma} Let $s$ be a node of the quartic $Q$ which is not a common point of $Q_{i}$'s. Then this point is the vertex of a singular member of the system~\eqref{eq6} which corresponds to a root of the polynomial $D(p:q)$ of multiplicity at least 2. Moreover, if $Q_{1}$ is this member then $s\in Q_{2}$.
\end{lemma}
\begin{proof} Since $s$ is a point of $Q$, it belongs to a single member of the system~\eqref{eq6}. We may assume that it is the quadric $Q_{1}$, so $(p:q)=(1:0)$. Since $s\in Q$ and the equation of $Q$ is $Q_{2}^{2}-4Q_{1}Q_{3}=0$, we see, that $s\in Q_{2}$ and $s\notin Q_{3}$. Thus if $s$ is a smooth point of $Q_{1}$ then it is a smooth point of $Q$, because linear terms of $Q$ and $Q_{1}$ are proportional with non-zero coefficient in a neighbourhood of $s$. Assume that $Q_{1}$ is singular along a line. If $Q_{1}$ is a double plane then $Q$ is singular along the whole curve $Q_{1}\cap Q_{2}$. If $Q_{1}$ is a pair of planes then the intersection of one of these planes with $Q$ is a double quadric, so the variety $X$ contains planes and this is not our case.

We may assume that $s=(1:0:0:0)$. Since $Q_{1}$ is singular at the point $s$, its equation contains no terms $x_{0}^{2}$ and $x_{0}x_{i}$. Also we know that $Q_{2}$ contains $s$, so its equation doesn't contain the term $x_{0}^2$. Hence the top left corner of the matrix $M(p:q)$ is divisible by $q^{2}$, while all elements in the top row and left column are divisible by $q$. This entailed that $D(p:q)$ is divisible by $q^{2}$, so $(1:0)$ is a multiple root of this polynomial.
\end{proof}
\begin{corollary} $D(p:q)$ has exactly four different roots of multiplicity $2$, and vertices of the corresponding quadrics are precisely $4$ missing nodes of the quartic~$Q$.
\end{corollary}
\begin{lemma} The vertices of four singular quadrics are in general position.
\end{lemma}
\begin{proof} Assume that three vertices are colinear. We may assume that quadrics $Q_{1}$, $Q_{3}$ and $Q_{1}+Q_{2}+Q_{3}$ are singular at the points $(1:0:0:0)$, $(0:1:0:0)$ and $(1:1:0:0)$ respectively. One can solve the system of equations on coefficients which arise from this singularities and see that all $Q_{i}$'s contain the line $x_{0}=x_{1}=0$, so the quartic $Q$ is singular along this line.

Assume that four vertices are coplanar. We may assume that quadrics $Q_{1}$, $Q_{3}$, $Q_{1}+Q_{2}+Q_{3}$ and $t^{2}Q_{1}+tQ_{2}+Q_{3}$ are singular at the points $(1:0:0:0)$, $(0:1:0:0)$, $(0:0:1:0)$ and $(1:1:1:0)$ respectively. One can solve the system of equations on coefficients that arises from this singularities and see that $Q|_{x_{3}=0}$ is a double quadric, so the variety $X$ contains planes and this is not our case.
\end{proof}

Assume that $(1:0), (0:1), (1:1)$ and $(t:1)$ are roots of $D(p:q)$ and $(1:0:0:0)$, $(0:1:0:0)$, $(0:0:1:0)$ and $(0:0:0:1)$ are vertices of the corresponding quadrics. We will call these points \emph{base nodes} of the quartic $Q$. Using the fact that $Q_{1}$ is singular at $(1:0:0:0)$, $Q_{3}$ is singular at $(0:1:0:0)$, $Q_{2}$ contains $(1:0:0:0)$ and $(0:1:0:0)$, $Q_{1}+Q_{2}+Q_{3}$ is singular at $(0:0:1:0)$ and $t^{2}Q_{1}+tQ_{2}+Q_{3}$ is singular at $(0:0:0:1)$, we obtain the following equations:

$$Q_{1}=a_{1}x_{1}^{2}+a_{2}x_{2}^{2}+a_{3}x_{3}^{2}+a_{12}x_{1}x_{2}+a_{13}x_{1}x_{3}+a_{23}x_{2}x_{3}$$
\begin{equation}\label{eq7}Q_{3}=b_{0}x_{0}^{2}+b_{2}x_{2}^{2}+b_{3}x_{3}^{2}+b_{02}x_{0}x_{2}+b_{03}x_{0}x_{3}+ta_{23}x_{2}x_{3}
\end{equation}
$$Q_{2}=(-a_{2}-b_{2})x_{2}^{2}-(ta_{3}+\frac{b_{3}}{t})x_{3}^{2}+c_{01}x_{0}x_{1}-b_{02}x_{0}x_{2}-\frac{b_{03}}{t}x_{0}x_{3}-a_{12}x_{1}x_{2}-ta_{13}x_{1}x_{3}-$$
$$-(1+t)a_{23}x_{2}x_{3}.$$

Now we introduce new coordinates $p'=p-tq, q'=p-q$. In these coordinates the point $(p:q)=(1:1)$ becomes $(1:0)$ and the point $(p:q)=(t:1)$ becomes $(0:1)$. In these coordinates our system can be written as follows $$p'^{2}Q_{1}'+p'q'Q_{2}'+q'^{2}Q_{3}'.$$
We know that $Q_{2}'$ contains points $(0:0:1:0)$ and $(0:0:0:1)$. This gives us following conditions on coefficients in the equation~\eqref{eq7}: $a_{2}=b_{2}$, $b_{3}=t^{2}a_{2}$. Also we can apply a diagonal transformation and get the following system:
$$Q_{1}=x_{1}^{2}+x_{2}^{2}+x_{3}^{2}+a_{12}x_{1}x_{2}+a_{13}x_{1}x_{3}+a_{23}x_{2}x_{3}$$
\begin{equation}\label{eq8}
Q_{3}=x_{0}^{2}+x_{2}^{2}+t^2x_{3}^{2}+b_{02}x_{0}x_{2}+b_{03}x_{0}x_{3}+ta_{23}x_{2}x_{3}
\end{equation}
$$Q_{2}=-2x_{2}^{2}-2tx_{3}^{2}+c_{01}x_{0}x_{1}-b_{02}x_{0}x_{2}-\frac{b_{03}}{t}x_{0}x_{3}-a_{12}x_{1}x_{2}-ta_{13}x_{1}x_{3}-(1+t)a_{23}x_{2}x_{3}.$$

Let us denote this system of equations by $\operatorname{Sys}(t; a_{12}, a_{13}, a_{23}, b_{02}, b_{03}, c_{01})$. Now our goal is to describe its automorphism group for different parameters.

Let $G'\subset \Aut(X)$ be a subgroup that preserves the pair of systems $\mathcal{F}_{1}$ and $-K_{X}-\mathcal{F}_{1}$. Then we may assume that either $G'=\Aut(X)$ or $G'$ is a subgroup of $\Aut(X)$ of index $3$ (if the index is $2$ then we can choose another pair of systems). The group $G'$ contains the Geiser involution, let $H$ be the quotient $G'/\Cyc_{2}$. Then $H$ is a subgroup of $\operatorname{PGL}_{4}(\operatorname{k})$ that preserves our system of quadrics. We have the natural short exact sequence
$$0\longrightarrow H'\longrightarrow H\longrightarrow H''\longrightarrow 0,$$
where $H'$ preserves all members of the system and $H''$ acts faithfully on $\PP^{1}$. The group $H''$ preserves the quadruple of points corresponding to singular members of the system of quadrics. In general situation the stabilizer of quadruple of points in $\PP^{1}$ is $\Cyc_{2}\times \Cyc_{2}$, but if $t=-1, \frac{1}{2}$ or $2$ then such stabilizer is $\Dih_{8}$ and if $t=\frac{1\pm \sqrt{3}i}{2}$ then such stabilizer is $\Alt_{4}$. The group $H'$ acts by changing signs of variables $x_{1}$, $x_{2}$ and $x_{3}$ and is isomorphic to $\Cyc_{2}^{n}$ for some $0\leq n\leq 3$.

\subsection{Calculation of the group $H''$}

\begin{lemma}\label{le9.5} Assume that $H''$ contains an element of order $3$. Then we have the following:
\begin{enumerate}
\item
In some coordinates we have $t=\frac{1+\sqrt{3}i}{2}$, $a_{12}=a_{13}=a_{23}$, $c_{01}=-b_{02}=-\frac{b_{03}}{t}$;
\item
The group $H''$ is isomorphic to $\Alt_{4}$ iff $a_{12}=-c_{01}$ and to $\Cyc_{3}$ otherwise.
\end{enumerate}
\end{lemma}
\begin{proof} Consider the map $$(p:q)\longmapsto \left(p+\frac{1}{t-1}q:\frac{1}{t-1}q\right),\ (x_{0}:x_{1}:x_{2}:x_{3})\longmapsto\left(\frac{x_{0}}{t}:x_{3}:x_{1}:x_{2}\right).$$ On $\PP^{1}$ it preserves $(1:0)$ and permutes cyclically three other points. It transforms our system $\operatorname{Sys}(t; a_{12}, a_{13}, a_{23}, b_{02}, b_{03}, c_{01})$ into $\operatorname{Sys}(t; a_{23}, a_{12}, a_{13}, -\frac{b_{03}}{t}, tc_{01}, b_{02})$. Now we also can change some signs of variables $x_{1}, x_{2}, x_{3}$ with the aim of getting the original system. But we need to consider only two cases: change all signs or don't change anything, since we can conjugate our map with diagonal map that changes some signs (i.e. we can consider our equations in another basis). In the first case we obtain the system
$$a_{12}=a_{13}=a_{23},\ b_{02}=-\frac{b_{03}}{t},\ b_{03}=tc_{01},\ c_{01}=b_{02}$$ that gives us $a_{12}=a_{13}=a_{23},\ b_{02}=b_{03}=c_{01}=0$, and in the second case we obtain the system
$$a_{12}=a_{13}=a_{23},\ b_{02}=\frac{b_{03}}{t},\ b_{03}=-tc_{01},\ c_{01}=-b_{02}$$ that gives us $a_{12}=a_{13}=a_{23},\ c_{01}=-b_{02}=-\frac{b_{03}}{t}$. This proves the first statement. Notice that the first case is a subcase of the second one.

Consider the map $$(p:q)\longmapsto \left(p-q:\frac{1}{t}p-q\right),\ (x_{0}:x_{1}:x_{2}:x_{3})\longmapsto\left(\frac{x_{3}}{t^2}:x_{2}:\frac{x_{1}}{t}:t x_{0}\right).$$
On $\PP^{1}$ it swaps $(1:0)$ with $(t:1)$ and $(0:1)$ with $(1:1)$. It transforms our system $\operatorname{Sys}(t; a, a, a, b, tb, -b)$ into $\operatorname{Sys}(t; a, b, b, a, tb, -a)$. Again we can change signs of $x_{1}, x_{2}$ and $x_{3}$ in the latter system, but all signs changes except the trivial one leads to $a=b=0$. The trivial one leads to $a=b$. This proves the second statement.
\end{proof}

\begin{remark} The proof of the previous statement also gives us an explicit expres\-sion of generators of the group $H$.
\end{remark}

\begin{lemma} Assume that $H''$ contains an element of order $4$. Then we have the following:
\begin{enumerate}
\item
In some coordinates we have $t=-1$, and either $a_{12}=a_{13}=b_{02}=b_{03},\ a_{23}=c_{01}=0$ or $a_{12}=a_{13}=b_{02}=-b_{03},\ a_{23}=-c_{01}$;
\item
The group $H''$ is isomorphic to $\Dih_{8}$.
\end{enumerate}
\end{lemma}
\begin{proof}Consider the map $$(p:q)\longmapsto (p+q:-p+q),\ (x_{0}:x_{1}:x_{2}:x_{3})\longmapsto\left(x_{2}:x_{3}:\frac{x_{1}}{2}:\frac{x_{0}}{2}\right).$$ On $\PP^{1}$ it cyclically permutes $(0:1), (1:1), (1:0)$ and $(-1:1)$. It transforms our system $\operatorname{Sys}(t; a_{12}, a_{13}, a_{23}, b_{02}, b_{03}, c_{01})$ into $\operatorname{Sys}(t; b_{02}, a_{12}, -c_{01}, b_{03}, a_{13}, a_{23})$. Now we also can change some signs of variables $x_{1}, x_{2}, x_{3}$ with the aim of getting the original system. But we need to consider only two cases: change all signs or don't change anything, since we can change the basis as in the previous lemma. In the first case we obtain the system
$$a_{12}=b_{02}=b_{03}=a_{13},\ a_{23}=-c_{01}=-a_{23},$$ and in the second case we obtain the system
$$a_{12}=b_{02}=-b_{03}=a_{13},\ a_{23}=-c_{01}.$$ This proves the first statement. Notice that the first case is not a subcase of the second one.

Consider the map $$(p:q)\longmapsto (-p:q),\ (x_{0}:x_{1}:x_{2}:x_{3})\longmapsto(-x_{0}:x_{1}:x_{3}:x_{2}).$$ On $\PP^{1}$ it preserves $(1:0)$ and $(0:1)$ and swaps $(1:1)$ and $(-1:1)$. It preserves the system $\operatorname{Sys}(t; a, a, b, a, -a, -b)$. Consider the map $$(p:q)\longmapsto (-p:q),\ (x_{0}:x_{1}:x_{2}:x_{3})\longmapsto(x_{0}:x_{1}:x_{3}:x_{2}).$$ It preserves the system $\operatorname{Sys}(t; a, a, 0, a, a, 0)$. This proves the second statement.
\end{proof}

\begin{lemma} Assume that $H''$ contains an element of order $2$ that fixes two base nodes, but doesn't contain an element of order $4$. Then we have the following:
\begin{enumerate}
\item
In some coordinates $t=-1$, and either (A) $a_{12}=a_{13},\ b_{02}=b_{03}\neq 0,\ c_{01}=0$ or (B) $a_{12}=a_{13},\ b_{02}=-b_{03}$;
\item
The group $H''$ is isomorphic to $\Cyc_{2}^{2}$ if $a_{12}=a_{23}=b_{02}=-b_{03}$.
\end{enumerate}
\end{lemma}
\begin{proof}Consider the map $$(p:q)\longmapsto (-p:q),\ (x_{0}:x_{1}:x_{2}:x_{3})\longmapsto(x_{0}:x_{1}:x_{3}:x_{2}).$$ On $\PP^{1}$ it preserves $(1:0)$ and $(0:1)$ and swaps $(1:1)$ with $(-1:1)$. It transforms our system $\operatorname{Sys}(t; a_{12}, a_{13}, a_{23}, b_{02}, b_{03}, c_{01})$ into $\operatorname{Sys}(t; a_{13}, a_{12}, a_{23}, b_{03}, b_{02}, -c_{01})$. We also can change signs of variables $x_{0}$ and $x_{2}$, and this gives us four systems of equations that give us two different cases from the first statement (other cases are their subcases).

Consider the map $$(p:q)\longmapsto (q:p),\ (x_{0}:x_{1}:x_{2}:x_{3})\longmapsto(x_{1}:x_{0}:x_{2}:x_{3}).$$ On $\PP^{1}$ it preserves $(1:1)$ and $(-1:1)$ and swaps $(0:1)$ with $(1:0)$. It transforms the system $\operatorname{Sys}(t; a, a, b, c, c, 0)$ into $\operatorname{Sys}(t; c, c, -b, a, a, 0)$ and $\operatorname{Sys}(t; a, a, b, c, -c, d)$ into $\operatorname{Sys}(t; c, -c, -b, a, $ $a, d)$. Also we can change signs of variables $x_{1}$, $x_{2}$ and $x_{3}$, and this gives us eight systems of equations in any case that give us two different cases from the second statement (all other cases are their subcases or subcases of the first statement in the previous lemma, i.e. $H'$ contains an element of order $4$).
\end{proof}

\begin{lemma} Assume that $H''$ contains an element of order $2$ that does not fix two base nodes, and the group $H''$ was not described in the previous lemmas. Then we have the following:
\begin{enumerate}
\item
In some coordinates either (A) $b_{03}=ta_{12},\ b_{02}=a_{13}$ or (B) $b_{03}=-ta_{12},$ $b_{02}=a_{13},$ $a_{23}=0,\ c_{01}=0$;
\item
The group $H''$ is isomorphic to $\Cyc_{2}^{2}$ if in some coordinates either (A) $b_{03}=ta_{12},\ b_{02}=a_{13},$ $a_{23}=-c_{01}$ or (B) $b_{03}=-ta_{12},\ b_{02}=a_{13},\ a_{23}=0,\ c_{01}=0$.
\end{enumerate}
\end{lemma}
\begin{proof}Consider the map $$(p:q)\longmapsto (tq:p),\ (x_{0}:x_{1}:x_{2}:x_{3})\longmapsto\left(x_{1}:\frac{x_{0}}{t}:x_{3}:\frac{x_{2}}{t}\right).$$ On $\PP^{1}$ it swaps $(1:0)$ with $(0:1)$ and $(1:1)$ with $(t:1)$. It transforms our system $\operatorname{Sys}(t; a_{12}, a_{13}, a_{23}, b_{02}, b_{03}, c_{01})$ into $\operatorname{Sys}(t; \frac{b_{03}}{t}, b_{02}, a_{23}, a_{13}, ta_{12}, c_{01})$. Also we can change signs of variables $x_{0}$ and $x_{2}$, and this gives us four systems of equations that give us two different cases from the first statement (other cases are their subcases): (A) occurs when we do not change signs and (B) occurs if we change both signs.

Consider the map $$(p:q)\longmapsto (tp-tq:p-tq),\ (x_{0}:x_{1}:x_{2}:x_{3})\longmapsto\left(x_{3}:-\frac{x_{2}}{t}:-\frac{x_{1}}{t-1}:\frac{x_{0}}{t^2-t}\right).$$ On $\PP^{1}$ it swaps $(1:1)$ with $(0:1)$ and $(t:1)$ with $(1:0)$. It transforms the system $\operatorname{Sys}(t; a, b, c, b, ta, d)$ into $\operatorname{Sys}(t; a, b, -d, b, ta, -c)$ and preserves the system $\operatorname{Sys}(t; a, b, 0, b, -ta, 0)$. Also we can change signs of variables $x_{1}$, $x_{2}$ and $x_{3}$, and this gives us eight systems of equations in the first case that give us single case from the second statement (all other cases are its subcases or can be reduced to it by changing variables).
\end{proof}

The previous lemmas give us complete description of the group $H''$.

\subsection{Calculation of the groups $H$ and $\widetilde{\Aut}(X)$ and birational rigidity}

\begin{propos}\label{pr8} Assume that $H'\simeq \Cyc_{2}^{3}$. Then we have the following:
\begin{enumerate}
\item
 $a_{12}=a_{13}=a_{23}=b_{02}=b_{03}=c_{01}=0$.
 \item
  The group $H''$ can be isomorphic to $\Alt_{4}$ (iff $t=\frac{1\pm \sqrt{3}\i}{2}$), $\Dih_{8}$ (iff $t=-1, 2$ or $\frac{1}{2}$) or $\Cyc_{2}^{2}$ (for other values of $t$).
\item
The group $H$ is a subgroup of the group $\widetilde{\Aut}(X)$ of index 3.
\item
The group $\widetilde{\Aut}(X)$ is isomorphic to $(\Cyc_{2}^{2}\rtimes \Sym_{4})\rtimes \Cyc_{3}$ (GAPId [288, 1025]) for $t=\frac{1\pm \sqrt{3}\i}{2}$, $\Cyc_{2}^{3}\rtimes \Sym_{4}$ (GAPId [192, 955]) for $t=-1, 2$ or $\frac{1}{2}$ and $\Cyc_{2}^{2}\rtimes \Sym_{4}$ (GAPId [96, 227]) for other values of $t$.
\item
The group $\Aut(X)$ is isomorphic to $\widetilde{\Aut}(X)\times \Cyc_{2}$.
\end{enumerate}
\end{propos}
\begin{proof}
The first statement is obvious. The second statement follows from the previous lemmas.

One can calculate that the matrix
$$\begin{pmatrix}
    \frac{1}{2}       & \frac{\alpha}{2} & -\frac{\beta}{2} & \frac{\alpha\beta}{2} \\
    \frac{1}{2\alpha}      & \frac{1}{2} & \frac{\beta}{2\alpha} & -\frac{\beta}{2} \\
    \frac{1}{2\beta}      &  -\frac{\alpha}{2\beta} & -\frac{1}{2} & -\frac{\alpha}{2} \\
    \frac{1}{2\alpha\beta}   & -\frac{1}{2\beta} & \frac{1}{2\alpha} & \frac{1}{2}
\end{pmatrix}$$
where $\alpha=\sqrt{-t}$ and $\beta=\sqrt{t-1}$, preserves the equation of the quartic $Q$ and obviously this map does not belong to $H$, since $H$ preserves the quadruple of base nodes. This proves the third statement. The fourth statement can be easily proved with GAP using the action of the group $\widetilde{\Aut}(X)$ on the set of nodes.

For the proof of the fifth statement notice that matrices of the generators of the group $\widetilde{\Aut}(X)$ together with $-\operatorname{Id}$ generate the subgroup of $\operatorname{GL}_{4}(\operatorname{k})$ that is a double extension of the group $\widetilde{\Aut}(X)$. Thus we have an embedding of the group $\widetilde{\Aut}(X)$ into an automorphism group of $\PP(2, 1, 1, 1, 1)$. Obviously the image of this embedding commutes with the Geiser involution, so the group $\Aut(X)$ is the direct product of $\widetilde{\Aut}(X)$ and $\Cyc_{2}$.
\end{proof}
\begin{propos}\label{pr9} Assume that $H'\simeq \Cyc_{2}^{2}$. Then we have the following:
\begin{enumerate}
\item
Exactly one of the numbers $a_{12}, a_{13}, a_{23}, b_{02}, b_{03}, c_{01}$ is non-zero. We may assume that $a_{23}\neq 0$ without loss of generality.
\item
The group $H''$ is isomorphic to $\Cyc_{2}^{2}$ (iff $t=-1$) or $\Cyc_{2}$ (for other values of $t$).
\item
The group $H$ coincides with $\widetilde{\Aut}(X)$.
\item
We always have an $\Aut(X)$-invariant pair of singular points so the variety $X$ cannot be $G$-birationally rigid.
\end{enumerate}
\end{propos}
\begin{proof} The first statement is obvious. The second statement follows from the previous lemmas.

 Assume that $t=-1$. We have $Q_{2}=-2x_{2}^{2}+2x_{3}^{2}=2(x_{3}-x_{2})(x_{3}+x_{2})$. Two planes $P_{1, 2}=\{x_{2}\pm x_{3}=0\}$ contain six singular points of $Q$ both (four common points of $Q_{i}$'s, $(1:0:0:0)$ and $(0:1:0:0)$). Quadruple of base nodes are distinguished from others since they lie on even number of planes. If $H$ is a subgroup of $\widetilde{\Aut}(X)$ of index $3$ then we have another such pair of planes, so there is a third plane $P$ containing $6$ singular points. This plane contains at least $2$ base nodes. Indeed, otherwise $P$ and $P_{1}$ or $P_{2}$ contain at least three common non-based nodes, but it is impossible. So the equation of $P$ is of the form $x_{i}=0$ or $x_{i}+\alpha x_{j}=0$. One can check that such a plane contains six nodes iff it is $P_{1}$ or $P_{2}$. The case $t\neq -1$ can be treated in the same way, but equations in this case looks more ugly.

 Since $H=\widetilde{G}$, $(1:0:0:0)$ and $(0:1:0:0)$ form an $\Aut(X)$-invariant pair of points.
\end{proof}
\begin{propos}\label{pr10} Assume that $H'\simeq \Cyc_{2}$. Then we have the following:
\begin{enumerate}
\item
 There are two cases : \newline(A) all coefficients in $\operatorname{Sys}(t; a_{12}, a_{13}, a_{23}, b_{02}, b_{03}, c_{01})$ with the same index are zero and we may assume without loss of generality that $b_{02}=b_{03}=c_{01}=0$; \newline (B) all coefficients are zero except two, such that the union of sets of their indices is the set $\{0, 1, 2, 3\}$ and we may assume without loss of generality that $a_{23}$ and $c_{01}\neq 0$.
 \item
 In the case (A) the group $H''$ can be isomorphic to $\Cyc_{3}$ (we may assume without loss of generality, that $t=\frac{1+\sqrt{3}\i}{2}$, $a_{12}=a_{13}=a_{23}\neq 0$), $\Cyc_{2}$ (we may assume without loss of generality, that $t=-1$, $a_{12}=a_{13}\neq 0$) or trivial.
 \item
 In the case (B) the group $H''$ is isomorphic to $\Dih_{8}$ iff $t=-1$ and $c_{01}=\pm a_{23}$ (we may assume that $c_{01}=a_{23}$), to $\Cyc_{2}^{2}$ iff either $t=-1$ and $c_{01}\neq \pm a_{23}$ or $a_{23}=\pm c_{01}$ with arbitrary $t\neq -1$ (we may assume that $a_{23}=c_{01}$) and to $\Cyc_{2}$ in other cases.
\item
In the case (A) we always have $\widetilde{\Aut}(X)$-invariant node or triple of nodes, thus $X$ is never $G$-birationally rigid.
\item
In the case (B) if $c_{01}\neq \pm a_{23}$ or $t\neq -1$ then $X$ is not $G$-birationally rigid.
\item
In the case (B) if $c_{01}=a_{23}$ and $t=-1$ then $\Aut(X)$ is isomorphic to $\Sym_{4}\times \Cyc_{2}^{2}$ if $c_{01}=\pm 2\sqrt{3}$ and to $\Dih_{8}\times \Cyc_{2}^{2}$ otherwise.
\end{enumerate}
\end{propos}
\begin{proof}
The first statement is obvious. The second and the third statement follows from previous lemmas.

In the case (A) the point $(1:0:0:0)$ is always $H$-invariant. If $H\subset\widetilde{\Aut}(X)$ is a subgroup of index $3$ then we have either an $\widetilde{\Aut}(X)$-invariant singular point, or an $\widetilde{\Aut}(X)$-invariant triple of points.

Assume that in the case (B) we have $t=-1$ and $c_{01}\neq \pm a_{23}$. In this case the group $H$ is generated by three maps
$$\pi: (x_{0}:x_{1}:x_{2}:x_{3})\longmapsto (-x_{0}:-x_{1}:x_{2}:x_{3}),$$
$$\sigma: (x_{0}:x_{1}:x_{2}:x_{3})\longmapsto (x_{1}:x_{0}:x_{2}:-x_{3})$$
$$\tau: (x_{0}:x_{1}:x_{2}:x_{3})\longmapsto (x_{0}:x_{1}:x_{3}:-x_{2})$$
and is isomorphic to $\Cyc_{2}^{3}$. Assume that $H$ is a subgroup of $\widetilde{\Aut}(X)$ of index $3$. Then we have three possibilities: $\widetilde{\Aut}(X)$ can be isomorphic to $\Cyc_{2}^{2}\times \Cyc_{6}$, $\Cyc_{2}\times \Alt_{4}$ or $\Cyc_{2}^{2}\times \Sym_{3}$. One can check by direct computations that $H$ acts on the set of singular points with three orbits of length $2, 2$ and $8$. In the first two cases $H$ is a normal subgroup of $\widetilde{\Aut}(X)$ of index $3$, thus every $\widetilde{\Aut}(X)$-orbit of singular points consists of one or three $H$-orbits of the same size. So we have an $\widetilde{\Aut}(X)$-invariant pair of singular points. In the third case we can consider the  normal subgroup $\Cyc_{2}^{2}\subset \widetilde{\Aut}(X)$. It acts with orbits of length either $1, 1, 2, 4, 4$ or $2, 2, 4, 4$. In both cases the quadruple of base nodes is distinguished, thus $H$ cannot be a subgroup of $\widetilde{\Aut}(X)$ of index $3$. Thus we have that $H=\widetilde{\Aut}(X)$ and there is a $G$-invariant pair of singular points.

Assume that in the case (B) we have $t\neq -1$ and $c_{01}\neq \pm a_{23}$. In this case the group $H$ is generated by two maps
$$\tau: (x_{0}:x_{1}:x_{2}:x_{3})\longmapsto (-x_{0}:-x_{1}:x_{2}:x_{3}),$$
 $$\sigma: (x_{0}:x_{1}:x_{2}:x_{3})\longmapsto (-x_{0}:x_{1}:x_{3}:x_{2}),$$ and is isomorphic to $\Cyc_{2}^{2}$. One can easily check that $\tau$ preserves $4$ singular points while $\sigma$ and $\tau\sigma$ don't preserve any singular point. So if $H\subset\widetilde{\Aut}(X)$ is a subgroup of degree $3$ then $\widetilde{\Aut}(X)\not\simeq \Alt_{4}$ (since all elements of degree $2$ in $\Alt_{4}$ are conjugate), so $\widetilde{\Aut}(X)$ is isomorphic to $\Dih_{12}$ or $\Cyc_{6}\times \Cyc_{2}$. In any case $X$ is not $G$-birationally rigid.

Assume that in the case (B) we have $a_{23}=c_{01}=c$ (we have no assumptions on $t$ here). Then we apply a change of coordinates given by the matrix
$$\begin{pmatrix}
    \frac{2i\delta}{\alpha\beta\gamma}        & -\frac{i\epsilon\beta}{2\alpha} & -\frac{\alpha\beta}{2\gamma} & -\frac{\epsilon\alpha}{2\beta\delta}\\
    -\frac{i\epsilon\beta}{\alpha\gamma\delta}       & \frac{i}{\alpha\beta} & \frac{-\alpha}{\beta\gamma\delta} & \frac{-\alpha\beta}{4\delta}\\
    -\frac{i\beta}{\alpha\gamma}       & \frac{i\epsilon}{\alpha\beta\delta} & \frac{\alpha}{\beta\gamma\delta} &  \frac{\epsilon\alpha\beta}{4\delta^2}\\
    \frac{2i}{\alpha\beta\gamma\delta}    & -\frac{i\beta}{2\alpha\delta} & \frac{\alpha\beta}{2\gamma\delta\epsilon} & \frac{\alpha}{2\beta\delta^2}
\end{pmatrix}$$
where $\alpha=\sqrt[4]{tc^2-4}, \beta=\sqrt{-c\epsilon+\alpha^2}$, $\gamma=\sqrt{c^2-4}$, $\delta=\sqrt{1-t}$ and $\epsilon=\sqrt{t}$. One can check what we transform the equation of $Q$ into the following:
$$-\frac{4(tc^2+4)}{c^2-4}x_{0}^2x_{1}^2-4x_{0}^2x_{2}^2-4x_{0}^2x_{3}^2-4x_{1}^2x_{2}^{2}-\frac{(c^2-4)^{2}t^2}{4(t-1)^2}x_{1}^{2}x_{3}^{2}+\frac{(tc^2+4)^{3}}{4(c^{2}-4)(t-1)^2}x_{2}^{2}x_{3}^{2}+$$
$$\frac{8\sqrt{1-t}c}{\sqrt{c^{2}-4}}x_{0}^{2}x_{2}x_{3}-\frac{2ct\sqrt{c^2-4}}{\sqrt{1-t}}x_{1}^{2}x_{2}x_{3}-\frac{8\sqrt{1-t}c}{\sqrt{c^{2}-4}}x_{2}^{2}x_{0}x_{1}+\frac{2ct\sqrt{c^2-4}}{\sqrt{1-t}}x_{3}^{2}x_{0}x_{1}-$$
$$-\frac{2c^2(tc^2+8t-4)}{c^2-4}x_{0}x_{1}x_{2}x_{3}.$$
One can check that this equation corresponds to the quartic obtained from system $\operatorname{Sys}(\frac{(c^2-4)t}{4(1-t)}; 0, $ $0, -2c\sqrt{\frac{1-t}{c^2-4}}, 0, 0, -2c\sqrt{\frac{1-t}{c^2-4}}\Large{)}$. This system is isomorphic to the original system $\operatorname{Sys}(t; 0, 0, c, 0, $ $0, -c)$ iff $c=\pm 2\sqrt{2-t}$. If $t=-1$ then the group $H$ is a subgroup of $\widetilde{\Aut}(X)$ of index $3$ and is isomorphic to $\Sym_{4}\times \Cyc_{2}$, and similarly to Lemma 10.12 we have that $\Aut(X)\simeq \Sym_{4}\times \Cyc_{2}^{2}$. If $t\neq -1$ then the group $H$ is a subgroup of $\widetilde{\Aut}(X)$ of index $2$ which is not our case. We also can apply the map $(x_{0}:x_{1}:x_{2}:x_{3})\mapsto (x_{1}:x_{0}:x_{2}:\frac{x_{3}}{t})$ and obtain the system $\operatorname{Sys}\left(\frac{4(1-t)}{(c^2-4)t}; 0, 0, -2c\sqrt{\frac{1-t}{c^2-4}}, 0, 0, -2c\sqrt{\frac{1-t}{c^2-4}}\right)$, but this system is never isomorphic to the original system. Obviously, all other permutations of coordinates lead either to one of these two systems or to a system with $a_{23}=0$ or $c_{01}=0$, so in this case $H$ is never a subgroup of $\widetilde{\Aut}(X)$ of index $3$. But if $\widetilde{\Aut}(X)\simeq H\simeq \Cyc_{2}^{3}$ then there is an $\widetilde{\Aut}(X)$-invariant line.
\end{proof}

\begin{propos}\label{pr11} Assume that $H'$ is trivial. Then we have the following:
\begin{itemize}
\item
$H''$ can be trivial or isomorphic to $\Cyc_{2}$, $\Cyc_{3}$, $\Cyc_{2}^{2}$ (two different cases), $\Dih_{8}$ or $\Alt_{4}$.
\item
If $H''$ is trivial or isomorphic to $\Cyc_{2}$, $\Cyc_{3}$, $\Cyc_{2}^{2}$ or $\Dih_{8}$ then $X$ is never $G$-biratio\-nally rigid.
\item
If $H''$ is isomorphic to $\Alt_{4}$ (recall that in this case we may assume that $t=\frac{1+\sqrt{3}i}{2}$, $a_{12}=a_{13}=a_{23}=-c_{01}=b_{02}=\frac{b_{03}}{t}$) then $\Aut(X)\simeq \Alt_{4}\times\Cyc_{6}$ if $a_{12}=1\pm\sqrt{3}$ and to $\Alt_{4}\times\Cyc_{2}$ otherwise.
\end{itemize}
\end{propos}
\begin{proof} If $H''$ is trivial then the group $\widetilde{\Aut}(X)$ is also trivial or isomorphic to $\Cyc_{3}$. In any case we have an $\Aut(X)$-invariant singular point or a triple of singular points. If $H''\simeq \Cyc_{3}$ then $\widetilde{\Aut}(X)\simeq \Cyc_{3}, \Cyc_{9}$ of $\Cyc_{3}^{2}$. In any case we have an $\Aut(X)$-invariant singular point or a triple of singular points. If $H''\simeq \Cyc_{2}$ then $\widetilde{\Aut}(X)\simeq \Cyc_{2}, \Cyc_{6}$ or $\Sym_{3}$. In any case $X$ cannot be $G$-birationally rigid.

Assume that $H$ is isomorphic to $\Cyc_{2}^{2}$. If $H=\widetilde{\Aut}(X)$ then $X$ is not $G$-birationally rigid, so we need to check if $H$ is a subgroup of $\widetilde{\Aut}(X)$ of index $3$. There are three groups of order $12$ containing $\Cyc_{2}^{2}$: $\Dih_{12}$, $\Cyc_{6}\times \Cyc_{2}$ and $\Alt_{4}$. In the first two cases $X$ cannot be $G$-birationally rigid, so we need to check only the last case. If $H$ acts on the set of base nodes with two orbits then nontrivial elements of $H$ cannot be conjugate (one can easily check that one of them acts on the set of nodes as a product of $6$ transpositions while two other elements act as a products of $5$ transpositions), thus $\widetilde{\Aut}(X)$ cannot be isomorphic to $\Alt_{4}$. So we need to check only the case of $\Cyc_{2}^{2}$ acting on the set of base nodes transitively.

Let us consider the case then $a_{12}=a\neq 0, b_{03}=-ta, b_{02}=a_{13}=b\neq 0, c_{01}=a_{23}=0$. In this case the group $H$ is generated by two maps $$\sigma:(x_{0}:x_{1}:x_{2}:x_{3})\longmapsto \left(\sqrt{t} x_{1}:-\frac{x_{0}}{\sqrt{t}}:\sqrt{t} x_{3}:-\frac{x_{2}}{\sqrt{t}}\right)$$ and $$\tau:(x_{0}:x_{1}:x_{2}:x_{3})\longmapsto \left(\sqrt{t}\sqrt{t-1} x_{3}:-\frac{\sqrt{t-1}x_{2}}{\sqrt{t}}:-\frac{\sqrt{t} x_{1}}{\sqrt{t-1}}:\frac{x_{0}}{\sqrt{t}\sqrt{t-1}}\right).$$ Maps $\sigma, \tau$ and $\sigma\tau$ act on $\PP^{3}$ with two lines of fixed points which form a tetrahedron with vertices $$p_{1}=(\sqrt{t}\sqrt{t-1}:-i\sqrt{t-1}:i\sqrt{t}:1), p_{2}=(-\sqrt{t}\sqrt{t-1}:i\sqrt{t-1}:i\sqrt{t}:1),$$ $$p_{3}=(\sqrt{t}\sqrt{t-1}:i\sqrt{t-1}:-i\sqrt{t}:1), p_{4}=(-\sqrt{t}\sqrt{t-1}:-i\sqrt{t-1}:-i\sqrt{t}:1).$$ Since $\widetilde{\Aut}(X)\simeq \Alt_{4}$ and $\Cyc_{2}^{2}\subset \Alt_{4}$ is a normal subgroup, the element of degree $3$ should preserve one point among $p_{i}$'s and permute cyclically three other points. Now we consider a map given by the matrix
$$\begin{pmatrix}
        \sqrt{t}\sqrt{t-1}    & -\sqrt{t}\sqrt{t-1} & \sqrt{t}\sqrt{t-1} & -\sqrt{t}\sqrt{t-1}\\
    -i\sqrt{t-1}       & i\sqrt{t-1} & i\sqrt{t-1} & -i\sqrt{t-1}\\
    i\sqrt{t}       & i\sqrt{t} & -i\sqrt{t}&  -i\sqrt{t}\\
    1    & 1 & 1 & 1
\end{pmatrix}$$
In new coordinates the equation of $Q$ has the form

$$(ib-\sqrt{t}a+2\sqrt{t-1})^2x_{0}^{4}+(-ib+\sqrt{t}a+2\sqrt{t-1})^2x_{1}^{4}+(ib+\sqrt{t}a+2\sqrt{t-1})^2x_{2}^{4}+$$
$$+(-ib-\sqrt{t}a+2\sqrt{t-1})^2x_{3}^{4}+\text{other terms.}$$
Assume that $ib-\sqrt{t}a+2\sqrt{t-1}=0$. Then one can check that an element of degree $3$ which acts like $(x_{0}:x_{1}:x_{2}:x_{3})\mapsto (x_{0}:\alpha x_{2}:\beta x_{3}:\gamma x_{1})$ preserve this equation only if $a=\frac{2\sqrt{t}}{\sqrt{t-1}}, b=\frac{-2i}{\sqrt{t-1}}$ and the quartic  $Q$ is a double quadric, which is impossible. The same is true if $-ib+\sqrt{t}a+2\sqrt{t-1}=0$, $ib+\sqrt{t}a+2\sqrt{t-1}=0$ or $-ib-\sqrt{t}a+2\sqrt{t-1}$. If all of these coefficients are nonzero we can apply a diagonal map and obtain an equation
$$x_{0}^{4}+x_{1}^{4}+x_{2}^{4}+x_{3}^{4}+2x_{0}^{2}x_{1}^{2}+2x_{2}^{2}x_{3}^{2}+\text{other terms},$$ and the automorphism of degree three acts iff $a^{2}+b^{2}=4$, but in this case $Q$ again is a double quadric. So this case is impossible.

Assume that $b_{03}=ta_{12}, b_{02}=a_{13}, c_{01}=-a_{23}$. In this case the group $H$ is generated by two maps $$\sigma:(x_{0}:x_{1}:x_{2}:x_{3})\longmapsto \left(\sqrt{t} x_{1}:\frac{x_{0}}{\sqrt{t}}:\sqrt{t} x_{3}:\frac{x_{2}}{\sqrt{t}}\right)\text{ and}$$ $$\tau:(x_{0}:x_{1}:x_{2}:x_{3})\longmapsto \left(\sqrt{t}\sqrt{t-1} x_{3}:-\frac{\sqrt{t-1}x_{2}}{\sqrt{t}}:-\frac{\sqrt{t} x_{1}}{\sqrt{t-1}}:\frac{x_{0}}{\sqrt{t}\sqrt{t-1}}\right).$$ Lines in $\PP^{3}$ preserved by the group $H$ form a family parametrized by the line $\PP^{1}$. The group $\widetilde{\Aut}(X)$ acts on this line with two fixed points, so there are two $\widetilde{\Aut}(X)$-invariant lines and $X$ is never $G$-birationally rigid.

Assume that $H$ is isomorphic to $\Dih_{8}$. If $H=\widetilde{\Aut}(X)$ then $X$ is not $G$-birationally rigid, so we need to check if the group $H$ is a subgroup of $\widetilde{\Aut}(X)$ of index $3$. There are four groups of order $24$ containing $\Dih_{8}$: $\Dih_{24}$, $\Sym_{4}$, $\Dih_{8}\times \Cyc_{3}$ and $\Cyc_{3}\rtimes \Dih_{8}$. If $\widetilde{\Aut}(X)$ is isomorphic to $\Dih_{24}$ or $\Dih_{8}\times \Cyc_{3}$ then $X$ is not $G$-birationally rigid. So we need to check two other possibilities. We may assume that $t=-1$ and we have two possibilities: either $a_{12}=a_{13}=b_{02}=b_{03}=a\neq 0, a_{23}=c_{01}=0$ or $a_{12}=a_{13}=b_{02}=-b_{03}=a\neq 0, a_{23}=-c_{01}=b$.

In the first case we can explicitly calculate all singularities of $X$: $$p_{1}=(1:0:0:0),\ p_{2}=(0:1:0:0),\ p_{3}=(0:0:1:0),\ p_{4}=(0:0:0:1),$$ $$p_{5}=(-a+\alpha:-a+\alpha:1:1),\ p_{6}=(-a+\alpha:-a-\alpha:1:1),$$
 $$p_{7}=(-a-\alpha:-a+\alpha:1:1),\ p_{8}=(-a-\alpha:-a-\alpha:1:1),$$
  $$p_{9}=(\sqrt{2}:-\sqrt{2}:-1:1),\ p_{10}=(-\sqrt{2}:\sqrt{2}:-1:1),$$
  $$p_{11}=(-a+\alpha:-a+\alpha:-1-2a(-a+\alpha):1),\ p_{12}=(-a-\alpha:-a-\alpha:2a(a+\alpha)-1:1),$$ where $\alpha=\sqrt{a^2-2}$. Two generators of the group $H$ act on the set of nodes as permutations $\sigma=(1, 4, 2, 3)(5, 8)(6, 12, 7, 11)(9, 10)$ and $\tau=(1,2)(6,7)(9,10)$. If $\widetilde{\Aut}(X)\simeq \Cyc_{3}\rtimes \Dih_{8}$ then the element $\sigma^{2}$ generates its center, but it acts on the set of nodes as $(1, 2)(3,4)(6,7)(11,12)$, so we have a distinguished quadruple of nodes which is impossible. If $\widetilde{\Aut}(X)\simeq \Sym_{4}$ then where is a normal subgroup $\sigma^{2}\in \Cyc_{2}^{2}\subset \Dih_{8}\subset \Sym_{4}$ and all elements of order $2$ in $\Cyc_{2}^{2}$ are conjugate to each other and thus have the same type of a cycle decomposition. But only $\sigma^{2}$ in $\Dih_{8}$ is a product of $4$ transpositions, all other are products of $3$ or $5$ transpositions. So this case is also impossible.

In the second case we again have a $1$-parametric family of $H$-invariant lines parametrized by $\PP^{1}$, and even if $H$ is a subgroup of $\widetilde{\Aut}(X)$ of index $3$ then there is an $\widetilde{\Aut}(X)$-invariant line in $\PP^{3}$.

Assume that $H$ is isomorphic to $\Alt_{4}$, i.e. $t=\frac{1+i\sqrt{3}}{2}$ and $a_{12}=a_{13}=a_{23}=b_{02}=\frac{b_{03}}{t}=-c_{01}=a$. Let us apply a map given by the matrix

$$\begin{pmatrix}
    \frac{\beta}{\alpha}        & \alpha^{3} & -\frac{1}{\alpha} & \frac{1}{t\alpha}\\
    \frac{1}{\alpha}       & \frac{\beta}{\alpha} & -\frac{i}{\alpha^{3}} & -\frac{1}{\alpha}\\
    -\alpha^{3}       & -\frac{1}{\alpha} & \frac{\beta}{\alpha} &  \frac{1}{t^2\alpha}\\
    t\alpha^{3}    & -\alpha^{3} & -\frac{1}{\alpha} & \frac{\beta}{\alpha}
\end{pmatrix}$$
where $\alpha=\sqrt[4]{t}$ and $$\beta=\frac{a(t+1)+\sqrt{a^2(t+1)^{2}-4t(a+1)}}{2}.$$ One can check that this map transforms our quartic to a quartic with equation corresponding to $\operatorname{Sys}(t; b, b, b, b, \frac{b}{t}, -b)$, where $$b=\frac{2a(8-8a+4(3-i\sqrt{3})\beta}{24a^{2}-40a-32-4(3-i\sqrt{3})a\beta}.$$ This system is isomorphic to the original one if $a=b$. One can check that this happens iff $a=-1$ or $1\pm\sqrt{3}$. But if $a=-1$ then the quadric $Q_{1}$ is a union of two planes which is impossible.

\end{proof}

\section{Case $18^{\circ}$}

In this case $X$ has a small $\QQ$-factorialization $\widetilde{X}$ which is the blow-up of a singular factorial cubic threefold $Y$ of Picard rank $1$ in a general smooth point $p$. The number of nodes on $Y$ is not greater than $5$, but it cannot be zero since $X$ is rational. By the Lemma~\ref{le1}, there are exactly four such points, and the corresponding singular points of $X$ form a single $G$-orbit (let us denote them by $p_1, p_2, p_3, p_{4}$). Remaining singular points of the variety $X$ comes from six lines passing through the point $p$. Their images lie on the unique trope $P$ and form a single $G$-orbit. Let us denote by $C$ the conic $P\cap Q$. The following lemmas 10.1--10.6 can be proven similarly to the corresponding lemmas of the case $20^\circ$ (see Lemmas 8.1--8.6).
\begin{lemma} There is a $G$-invariant point $q$ in $\PP^{3}$.
\end{lemma}

There is the natural short exact sequence $$0\longrightarrow G'\longrightarrow G\longrightarrow G''\longrightarrow 0,$$ where the group $G'$ acts trivially on $P$ and $G''$ acts faithfully on $P$.
\begin{lemma} The group $G'$ is trivial.
\end{lemma}
\begin{lemma} The group $G''$ is isomorphic to $\Alt_{4}$ or $\Sym_{4}$.
\end{lemma}
\begin{corollary} The group $G$ is isomorphic to $\Alt_{4}$ or $\Sym_{4}$.
\end{corollary}
\begin{lemma} $\PP^{3}$ with the action of the group $\widetilde{G}$ is the projectivisation of a $4$-dimensional representation of the group $G$.
\end{lemma}
\begin{lemma} The equation of the quartic $Q$ is symmetric in some coordinates.
\end{lemma}

\begin{propos}\label{pr12} The quartic $Q$ can be given by the equation
\begin{equation}\label{eq4}
s_{2}^{2}-2s_{1}s_{3}+s_{1}^2s_{2}+A(2s_{1}s_{3}-3s_{1}^2s_{2}+s_{1}^{4})=0
\end{equation}
for some $A\neq 1$. The group $\Aut(X)$ is isomorphic to $\Sym_{4}\times \Cyc_{2}$ and the group $G$ is $\Sym_{4}\times \Cyc_{2}$, $\Alt_{4}\times \Cyc_{2}$ or $\Sym_{4}$ (which permutes planes on $X$).
\end{propos}
\begin{proof} We know from the previous lemma that the quartic $Q$ has a symmetric equation. Since the intersection $Q\cap P$ is a double conic, the equation of $X$ has the form~\eqref{eq2}. The variety $X$ is singular at the orbit of length $4$ outside of $P$. Thus $X$ is singular at the point $(a:a:a:b)$ for some $a$ and $b$. Similarly to Lemma 9.6 we may assume that $a=0$ and $b=1$. One can easily calculate that the quartic $Q$ has the equation of the required form.
\end{proof}

\section{Case $17^{\circ}$}

\begin{propos} In this case $X$ is never $G$-birationally rigid.
\end{propos}
\begin{proof}
In this case $X$ contains exactly $11$ nodes. By the Lemma~\ref{le1} either they form a unique $G$-orbit or they form two orbits of length 4 and 7 respectively.

Assume that they form a unique $G$-orbit. Thus the group $G$ contains an element of order $11$ which acts transitively on the set of nodes. We may assume that such an element acts diagonally: $$\sigma: (x_{0}:x_{1}:x_{2}:x_{3})\longmapsto (x_{0}:\xi x_{1}:\xi^{a} x_{2}:\xi^{b} x_{3})$$ where $\xi$ is a root of unity of degree $11$ and $0\leq a, b\leq 10$ are integers. Since all nodes cannot be coplanar, we see, that $a$ and $b\neq 0, 1$ and $a\neq b$ and we may assume that one singular point has coordinates $(1:1:1:1)$. Note, that we can perform three following operations:
\begin{enumerate}
\item
to choose another root of unity;
\item
to multiply all coordinates by the same number;
\item
to reorder coordinates.
\end{enumerate}
Using this operations we can reduce the task to four cases: $(a, b)=(2, 3)$, $(2, 4)$, $(2, 5)$ or $(3, 4)$. Then in any case we do the following: for every $0\leq c\leq 10$ we write down all monomials in variables $x_{i}$ such that $\sigma$ acts of them by multiplication by $\xi^{c}$, then we take the sum of such monomials with undefined coefficients and solve the system of linear equations which comes from the fact that $(1:1:1:1)$ is a singular point. In any case except one either the system has no nonzero solutions or the resulting quartic is reducible or non-reduced. The only exception is the case $(a, b, c)=(2, 3, 6)$. In this case we have the following monomials: $x_{0}^{2}x_{3}^{2}, x_{0}x_{1}x_{2}x_{3}, x_{0}x_{2}^{3}, x_{1}^{3}x_{3}, x_{1}^{2}x_{2}^{2}$. Solving the system gives us a one-parametric family of quartics $$\alpha(x_{0}^{2}x_{3}^{2}+3x_{1}^{2}x_{2}^{2}-2x_{0}x_{2}^{3}-2x_{1}^{3}x_{3})+\beta(x_{0}x_{3}-x_{1}x_{2})^2=0.$$
But in this case the quartic $Q$ is singular along the whole twisted cubic $(t^{3}:t^{2}s:ts^{2}:s^{3})$, so this case is impossible.

Assume that we have two orbits of length 4 and 7. In this case the group $G$ contains an element of order $7$ which acts transitively on the second orbit of nodes and trivially on the first orbit. We may assume that such an element acts diagonally: $$\sigma: (x_{0}:x_{1}:x_{2}:x_{3})\longmapsto (x_{0}:\xi x_{1}:\xi^{a} x_{2}:\xi^{b} x_{3})$$ where $\xi$ is a root of unity of degree $7$ and $0\leq a, b\leq 6$ are integers. Since all nodes cannot be coplanar, we see, that $a$ and $b\neq 0, 1$ and $a\neq b$ and we may assume that one singular point has coordinates $(1:1:1:1)$. Since four remaining nodes are fixed by $\sigma$, they coincide with points $(1:0:0:0), (0:1:0:0), (0:0:1:0)$ and $(0:0:0:1)$. Thus we must consider only monomials which has degree at most $2$ in any variable. Using the same arguments as in the previous case we see, that in any case either the system has no nonzero solutions or the resulting quartic is reducible or non-reduced.
\end{proof}
\section{Case $16^{\circ}$}

As in the case $19^{\circ}$ we know, that there is a system of quadrics
\begin{equation}\label{eq16.1}p^{2}Q_{1}+pqQ_{2}+q^{2}Q_{3}=0
 \end{equation}
 parametrized by $\PP_{1}$ such that $Q_{i}$'s intersect transversally in eight points and the equation of the quartic $Q$ is the discriminant $Q_{2}^{2}-4Q_{1}Q_{3}=0$, but in this case such system of quadrics is unique. We have eight singular points on the quartic $Q$ which are intersection points of $Q_{i}$'s and at most three additional singular points. Thus by Lemma~\ref{le1} there are no additional singular points. Thus we have exactly $8$ singular members in our system of quadrics. Let us denote by $(p_{i}: q_{i}), i=1, ..., 8$ the point on $\PP^{1}$ corresponding to a singular quadric and by $v_{i}$ the vertice of corresponding quadric. By $Q_{(p_{0}:q_{0})}$ we will denote the quadric in the system~\eqref{eq16.1} for $(p:q)=(p_{0}:q_{0})$.

We have the natural short exact sequence

$$0\longrightarrow G'\longrightarrow \widetilde{\Aut}(X)\longrightarrow G''\longrightarrow 0,$$

where the group $G'$ acts trivially on the $\PP^{1}$ parametrizing the system of quadrics and $G''$ is a subgroup of $\operatorname{PGL}_{2}(\operatorname{k})$.

\begin{lemma} Either the group $G'$ is trivial or there are two skew lines such that both of them contain four points $v_{i}$.
\end{lemma}
\begin{proof} Assume that the group $G'$ is nontrivial. Let $g\in G'$ be an arbitrary nontrivial element. Then we have the following possibilities for the fixed locus $\operatorname{Fix}(g)$: (i) four points; (ii) two points and a line; (iii) two skew lines; (iv) a point and a plane. The first case is impossible since all points $v_{i}$ lie in $\operatorname{Fix}(g)$. The second case is impossible since in this case we have an $\widetilde{\Aut}(X)$-invariant line.

In the fourth case the plane $P$ contains seven or eight $v_{i}$'s. Consider the induced system of quadrics $p^{2}Q_{1}|_{P}+pqQ_{2}|_{P}+q^{2}Q_{3}|_{P}=0$. If its determinant is a nonzero polynomial then it has degree $6$, so this system contains at most six singular members, this is contradiction. If the determinant is zero polynomial then the locus of degenerate conics in the linear system $\alpha Q_{1}|_{P}+\beta Q_{2}|_{P}+\gamma Q_{3}|_{P}=0$ is a union of a conic and a line, so we have $G''$-invariant conic or pair of conics, thus $G''$ is a cyclic or dihedral group. Since $G'$ acts trivially on $P$ (otherwise we have too many $v_{i}$'s on one line and this line is $\widetilde{G}$-invariant), we have $G''$-invariant (and thus $\widetilde{\Aut}(X)$-invariant) line on $P$.

In the third case both lines contain four $v_{i}$'s, otherwise they are $\widetilde{\Aut}(X)$-invariant.
\end{proof}

\begin{lemma} Assume that $G'$ is trivial. Then the group $\widetilde{\Aut}(X)$ is isomorphic to $\Alt_{4}$ or $\Sym_{4}$.
\end{lemma}
\begin{proof} We know that in this case $\widetilde{\Aut}(X)=G''$. It cannot be cyclic or dihedral. Also it cannot be isomorphic to $\Alt_{5}$ since there are no $\Alt_{5}$-orbits of length less or equal to $8$.
\end{proof}
\begin{lemma} Assume that $G'$ is trivial and $H\simeq \Alt_{4}$ is a subgroup of $\widetilde{\Aut}(X)$. Then $v_{i}$'s form two $H$-orbits of length four and points in one orbit are colinear.
\end{lemma}
\begin{proof} The group $\Alt_{4}$ acts faithfully on $\PP_{1}$ and there is only one $\Alt_{4}$-invariant set of eight points -- the union of two orbits of length 4. This proves the first statement.

Assume that $v_{i}$'s in one orbit are in general position. We may assume that $(p_{1}:q_{1})=(1:0)$, $(p_{2}:q_{2})=(0:1)$, $(p_{3}:q_{3})=(1:1)$, $(p_{4}:q_{4})=(\alpha:1)$, $(p_{5}:q_{5})=(1:\beta+1)$, $(p_{6}:q_{6})=(\alpha+1:1)$, $(p_{7}:q_{7})=(\alpha^{2}:1)$, $(p_{8}:q_{8})=(\beta:1)$, $v_{1}=(1:0:0:0)$, $v_{2}=(0:1:0:0)$, $v_{3}=(0:0:1:0)$ and $v_{4}=(0:0:0:1)$, there $\alpha=\frac{1+i\sqrt{3}}{2}$ and $\beta=\frac{1-i\sqrt{3}}{2}$. In this case we can explicitly describe the action of $H$ on $\PP^{3}$. Up to conjugation by a diagonal map it is generated by the maps $$(x_{0}:x_{1}:x_{2}:x_{3})\longrightarrow (x_{0}:x_{2}:x_{3}:x_{1})$$ and
$$(x_{0}:x_{1}:x_{2}:x_{3})\longrightarrow (x_{1}:x_{0}:\gamma x_{3}:\gamma^{-1}x_{2}),$$ where $\gamma^{4}=1$ (the last equality holds because points $v_{i}$ don't lie on $Q$, so its equation contains terms $x_{0}^{4}, x_{1}^{4}, x_{2}^{4}$ and $x_{3}^{4}$). In any case one can check all possibilities for $v_{5}, v_{6}, v_{7}$ and $v_{8}$ (they form another $\Alt_{4}$-orbit) and solve with computer the system of equation on coefficients of the quadrics $Q_{1}$, $Q_{2}$ and $Q_{3}$ which we obtain from the fact that $Q_{(p_{i}:q_{i})}$ is singular at the point $v_{i}$. It turns out that such system either has no nonzero solutions or gives us an equation of $Q$ of the form $C(x_{0}x_{1}+x_{0}x_{2}+x_{0}x_{3}+x_{1}x_{2}+x_{1}x_{3}+x_{2}x_{3})^{2}=0$ which is forbidden.
\end{proof}
\begin{corollary} There are two skew lines (one may assume that they are $x_{0}=x_{1}=0$ and $x_{2}=x_{3}=0$) containing four $v_{i}$'s.
\end{corollary}
\begin{lemma} Assume that $(p_{1}:q_{1})=(1:0)$, $(p_{2}:q_{2})=(0:1)$, $(p_{3}:q_{3})=(1:1)$, $(p_{4}:q_{4})=(a:1)$, $v_{1}=(1:0:0:0)$, $v_{2}=(0:1:0:0)$, $v_{3}=(1:1:0:0)$ and $v_{4}=(b:1:0:0)$. Then $a\neq b$.
\end{lemma}
\begin{proof} One can solve the corresponding system of equation on coefficients of $Q_{i}$'s and see that if $a=b$ then $Q$ contains $v_{1}, v_{2}, v_{3}$ and $v_{4}$, this is impossible.
\end{proof}
\begin{lemma} The group $G''$ is isomorphic to $\Sym_{4}$ or $\Dih_{8}$.
\end{lemma}
\begin{proof} Let $H\subset G''$ be a subgroup of index two which preserves the quadruple of points $(p_{1}:q_{1}),(p_{2}:q_{2}), (p_{3}:q_{3}), (p_{4}:q_{4})$. This is a subgroup of $\Alt_{4}$, $\Dih_{8}$ or $\Cyc_{2}^{2}$. This group automatically acts on the line $x_{0}=x_{1}=0$. Since $a\neq b$, $H$ cannot contain an element of degree four or an element of degree two which preserves two points $(p_{i}:q_{i})$. Also $H$ cannot be a cyclic group, otherwise we have a $H$-invariant point on the line $x_{0}=x_{1}=0$ and thus a $\widetilde{\Aut}(X)$-invariant line in $\PP^{3}$. Thus $H\simeq \Cyc_{2}^{2}$ (in this case $G''\simeq \Dih_{8}$) or $\Alt_{4}$ (in this case $G''\simeq \Sym_{4}$).
\end{proof}

\begin{propos}\label{pr13} If the group $G''$ is isomorphic to $\Dih_{8}$ then there is a two-para\-metric family of possible varieties and their automorphism groups are isomorphic to $\Dih_{8}\times \Cyc_{2}^{2}$.
\end{propos}
\begin{proof} We may assume that  $(p_{1}:q_{1})=(1:0)$, $(p_{2}:q_{2})=(0:1)$, $(p_{3}:q_{3})=(\alpha:1)$, $(p_{4}:q_{4})=(1:1)$, $(p_{5}:q_{5})=(\alpha:\alpha-\sqrt{\alpha-\alpha^2})$, $(p_{6}:q_{6})=(\alpha-\sqrt{\alpha-\alpha^2}:1)$, $(p_{7}:q_{7})=(\alpha:\alpha+\sqrt{\alpha-\alpha^2})$, $(p_{8}:q_{8})=(\alpha+\sqrt{\alpha-\alpha^2}:1)$, $v_{1}=(1:0:0:0)$, $v_{2}=(0:0:1:0)$, $v_{3}=(0:1:0:0)$, $v_{4}=(0:0:0:1)$, $v_{5}=(1:1:0:0)$, $v_{6}=(0:0:1:1)$, $v_{7}=(\beta:1:0:0)$ and $v_{8}=(0:0:\beta:1)$, where $\alpha\neq 0, 1, \frac{1}{2}$, $\beta\neq 0, 1$ and $\beta\neq\frac{2\sqrt{\alpha-\alpha^2}-1}{2\sqrt{\alpha-\alpha^2}+1}$ (such points are more convenient from computational point of view). The group $\Dih_{8}$ generated by
$$\sigma: (p:q)\longmapsto (\alpha q:p) \text{ and } \tau:(p:q)\longmapsto ((\alpha+\sqrt{\alpha-\alpha^2})p-aq:p-(\alpha-\sqrt{\alpha-\alpha^2})q)$$ acts transitively on the set of $(p_{i}:q_{i})$'s. We have the induced action on $v_{i}$'s: $\sigma$ transpose $v_{2i-1}$ with $v_{2i}$ and $\tau$ permutes them in the following way: $v_{1}\to v_{6}\to v_{3}\to v_{8}\to v_{1}$, $v_{2}\to v_{7}\to v_{4}\to v_{5}\to v_{2}$. One can solve the corresponding system of equations and obtain the single solution up to multiplication by a scalar:
\begin{equation}
\begin{gathered}
Q_{1}=(2a-1)(b-1)bx_{1}^{2}+\frac{2a(a-1+b\sqrt{a-a^2})-(b-1)\sqrt{a-a^2}}{a-\sqrt{a-a^2}}x_{2}^{2}+\\
+b(a(b-1)+\sqrt{a-a^2}(b+1))x_{3}^{2}-4b\sqrt{a-a^2}x_{2}x_{3},\\
Q_{2}=-(2(a-a^2)(b-1)+\sqrt{a-a^2}(b+1))x_{0}^{2}-b(2a^2(b-1)+\\
+\sqrt{a-a^2}(b+1))x_{1}^{2}-(2(a-a^2)(b-1)+\sqrt{a-a^2}(b+1))x_{2}^{2}-b(2a^2(b-1)+\\
+\sqrt{a-a^2}(b+1))x_{3}^{2}-4b\sqrt{a-a^2}x_{0}x_{1}-4b\sqrt{a-a^2}x_{2}x_{3},\\
Q_{3}=a\frac{2a(a-1+b\sqrt{a-a^2})-(b-1)\sqrt{a-a^2}}{a-\sqrt{a-a^2}}x_{0}^{2}+ab(a(b-1)+\\
+\sqrt{a-a^2}(b+1))x_{1}^{2}+a(2a-1)(b-1)bx_{3}^{2}-4ab\sqrt{a-a^2}x_{0}x_{1}.
\end{gathered}
\end{equation}
One can easily see that $Q$ is preserved by the maps $(x_{0}:x_{1}:x_{2}:x_{3})\mapsto(x_{0}:x_{1}:-x_{2}:-x_{3})$ (which generates $G'$) and $(x_{0}:x_{1}:x_{2}:x_{3})\mapsto(x_{2}:x_{3}:x_{0}:x_{1})$. More difficult but direct computation shows that it is preserved by the map given by the matrix
$$\begin{pmatrix}
    0        & 0 & \frac{\sqrt{b}}{\sqrt{b-1}} & -\frac{\sqrt{b}}{\sqrt{b-1}}\\
    0     & 0 & \frac{1}{\sqrt{b}\sqrt{b-1}} & -\frac{\sqrt{b}}{\sqrt{b-1}}\\
           \frac{i}{\sqrt{b-1}}& \frac{-ib}{\sqrt{b-1}} & 0 &  0\\
     \frac{i}{\sqrt{b-1}}& \frac{-i}{\sqrt{b-1}} & 0 & 0
\end{pmatrix}$$

Obviously, last two maps act on $v_{i}$'s as generators of the group $\Dih_{8}$, so these three maps generate $\widetilde{\Aut}(X)$. One can check that they together with the Geiser involution generate the group $\Dih_{8}\times \Cyc_{2}^{2}$.
\end{proof}
\begin{propos}\label{pr14} If the group $G''$ is isomorphic to $\Sym_{4}$ then there is unique variety with the automorphism group $\Sym_{4}\times \Cyc_{2}^{2}$.
\end{propos}
\begin{proof} We may assume that  $(p_{1}:q_{1})=(1:0)$, $(p_{2}:q_{2})=(\alpha^2:1)$, $(p_{3}:q_{3})=(0:1)$, $(p_{4}:q_{4})=(\alpha+1:1)$, $(p_{5}:q_{5})=(1:1)$, $(p_{6}:q_{6})=(\frac{1}{2-\alpha}:1)$, $(p_{7}:q_{7})=(\alpha:1)$, $(p_{8}:q_{8})=(\beta:1)$, $v_{1}=(1:0:0:0)$, $v_{2}=(0:0:1:0)$, $v_{3}=(0:1:0:0)$, $v_{4}=(0:0:0:1)$, $v_{5}=(1:1:0:0)$, $v_{6}=(0:0:1:1)$, $v_{7}=(\beta:1:0:0)$ and $v_{8}=(0:0:\beta:1)$, where $\alpha=\frac{1\pm i\sqrt{3}}{2}$, $\beta=\frac{1}{\alpha}$. The group $\Sym_{4}$ generated by $$\sigma: (p:q)\longmapsto (\alpha^2 p+(1-\alpha^2)q:p-\alpha^2 q) \text{ and } \tau:(p:q)\longmapsto ((\alpha-1)p+q:q)$$ acts transitively on the set of $(p_{i}:q_{i})$'s. We have the induced action on $v_{i}$'s: $\sigma$ transpose $v_{2i-1}$ with $v_{2i}$ and $\tau$ permutes them in the following way: $v_{1}\to v_{1}$, $v_{3}\to v_{5}\to v_{7}\to v_{3}$, $v_{2}\to v_{8}\to v_{4}\to v_{2}$, $v_{6}\to v_{6}$. One can solve the corresponding system of equations and obtain the single solution up to multiplication by a scalar:
\begin{equation}
\begin{gathered}
Q_{1}=x_{1}^{2}+(x_{3}-x_{2})(x_{3}+\beta x_{2}),\\
Q_{2}=\alpha^2 x_{0}^{2}+x_{2}((1+\alpha) x_{3}+\alpha x_{2}),\\
Q_{3}=-2\sqrt{3}ix_{0}^2-(1+\alpha)x_{1}^{2}+2\alpha x_{0}x_{1}+x_{2}^{2}-(1+\alpha)x_{3}^{2}+2\alpha^{2}x_{2}x_{3}.
\end{gathered}
\end{equation}
One can easily see that the quartic $Q$ is preserved by the map $$(x_{0}:x_{1}:x_{2}:x_{3})\mapsto(x_{0}:x_{1}:-x_{2}:-x_{3})$$ which generates $G'$. Also the equation of $Q$ is preserved by the map $$(x_{0}:x_{1}:x_{2}:x_{3})\mapsto(x_{2}:x_{3}:x_{0}:x_{1})$$. More difficult but direct computation shows that it is preserved by the map given by the matrix
$$\begin{pmatrix}
    1     & -\beta & 0 & 0\\
    0     & -\beta & 0 & 0\\
    0     &0 & \beta &  \alpha\\
     0& 0 & 1 & 0
\end{pmatrix}$$

Obviously, last two maps act on $v_{i}$'s as generators $\tau, \sigma$ of the group $\Sym_{4}$, so these three maps generate $\widetilde{\Aut}(X)$. One can check that they together with the Geiser involution generate the group $\Sym_{4}\times \Cyc_{2}^{2}$.
\end{proof}

\end{document}